\documentclass[a4paper,11pt]{article}
\usepackage{a4wide}
\usepackage{authblk}
\usepackage{amsmath,amsthm,amssymb}
\usepackage[utf8]{inputenc}
\usepackage[swedish,english]{babel}
\usepackage{graphicx}
\usepackage{caption}
\usepackage{subcaption}
\usepackage{sidecap, caption}
\usepackage{tikz}
\usepackage[thinc]{esdiff}
\usepackage[pdftitle={Morphological stability for in silico models of
  avascular tumors},
pdfauthor={Erik Blom, Stefan Engblom},
pdffitwindow=true,
breaklinks=true,
colorlinks=true,
urlcolor=blue,
linkcolor=red,
citecolor=red,
anchorcolor=red]{hyperref}
\usepackage{csquotes}
\usepackage{orcidlink}  
\usepackage[backend=biber,style=numeric,sorting=nyt,sortcites,firstinits=true]{biblatex}

\usepackage{lineno}         % numbers lines
\numberwithin{equation}{section}
\numberwithin{table}{section}
\numberwithin{figure}{section}

\theoremstyle{plain}
\newtheorem{theorem}{Theorem}[section]
\newtheorem{lemma}[theorem]{Lemma}
\newtheorem{proposition}[theorem]{Proposition}

\theoremstyle{definition}
\newtheorem{definition}{Definition}[section]
\newtheorem{assumption}[definition]{Assumption}
\newcommand{\figref}[1]{Fig.~\ref{fig:#1}}
\newcommand{\tabref}[1]{Tab.~\ref{tab:#1}}

\newcommand{\Ncells}{N_{\mathrm{cells}}}
\newcommand{\Nvox}{N_{\mathrm{vox}}}

\newcommand{\muprol}{\mu_{\mathrm{prol}}}
\newcommand{\mudeath}{\mu_{\mathrm{death}}}
\newcommand{\mudeg}{\mu_{\mathrm{deg}}}

\newcommand{\muprolbar}{\bar{\mu}_{\mathrm{prol}}}
\newcommand{\mudeathbar}{\bar{\mu}_{\mathrm{death}}}
\newcommand{\lambdabar}{\bar{\lambda}}

\newcommand{\kappaprol}{\kappa_{\mathrm{prol}}}
\newcommand{\kappadeath}{\kappa_{\mathrm{death}}}
\newcommand{\Kprol}{K_{\mathrm{prol}}}
\newcommand{\Kdeath}{K_{\mathrm{death}}}

\newcommand{\rpeq}{r_p^{\mathrm{eq}}}
\newcommand{\rqeq}{r_q^{\mathrm{eq}}}
\newcommand{\rneq}{r_n^{\mathrm{eq}}}

\newcommand{\cout}{c_{\mathrm{out}}}
\newcommand{\pext}{p^{(\mathrm{ext})}}

\newcommand{\Dout}{D_{\mathrm{ext}}}

\newcommand{\interface}[1]{r_{#1}}
\newcommand{\subquant}[2]{#1^{(#2)}} % quantity #1 in subspace #2

\newcommand{\pert}[2]{\tilde{#1}^{(#2)}} % quantity perturbed
\newcommand{\rpert}[1]{\tilde{r}_{#1}} % interface perturbed

\newcommand{\rcoef}[1]{\alpha_k^{(#1)}}

\newcommand{\ppert}[1]{\tilde{p}^{(#1)}}
\newcommand{\pcoef}[1]{\gamma_k^{(#1)}}
\newcommand{\pextcoef}{\gamma_k^{(\mathrm{ext})}}

\newcommand{\cpert}[1]{\tilde{c}^{(#1)}}
\newcommand{\ccoef}[1]{\beta_k^{(#1)}}

\newcommand{\cextpert}{\tilde{c}^{(\mathrm{ext})}}
\newcommand{\cextcoef}{\beta_k^{(\mathrm{ext})}}

\newcommand{\Omegaext}{\Omega_{\mathrm{ext}}}

\newcommand{\Ordo}[1]{\mathcal{O}\left(#1\right)}

\newcommand{\review}[1]{#1}

\title{Morphological stability for \textit{in silico} models of
  avascular tumors}

\author[1]{Erik Blom\orcidlink{0009-0005-8141-6802}}

\author[1]{Stefan Engblom\thanks{\textit{Corresponding author.}  URL:
    \url{http://user.it.uu.se/~stefane}, telephone +46-18-471 27 54,
    fax +46-18-51 19 25.}\orcidlink{0000-0002-3614-1732}}

\affil[1]{{\footnotesize Division of Scientific Computing, Department
    of Information Technology, Uppsala University, SE-751 05 Uppsala,
    Sweden. E-mail: \href{mailto:erik.blom@it.uu.se}{erik.blom},
    \href{mailto:stefane@it.uu.se}{stefane@it.uu.se}.}}

\date{\today}

\begin{document}

\selectlanguage{english}

\maketitle

\begin{abstract}
  % opening - quest for data-driven modeling
  The landscape of computational modeling in cancer systems biology is
  diverse, offering a spectrum of models and frameworks, each with its
  own trade-offs and advantages. Ideally, models are meant to be
  useful in refining hypotheses, to sharpen experimental procedures
  and, in the longer run, even for applications in personalized
  medicine. One of the greatest challenges is to balance model realism
  and detail with experimental data to eventually produce useful
  data-driven models.

  % our work
  We contribute to this quest by developing a transparent, highly
  parsimonious, first principle \textit{in silico} model of a growing
  avascular tumor. We initially formulate the physiological
  considerations and the specific model within a stochastic cell-based
  framework. We next formulate a corresponding mean-field model using
  partial differential equations which is amenable to mathematical
  analysis. Despite a few notable differences between the two models,
  we are in this way able to successfully detail the impact of all
  parameters in the stability of the growth process and on the
  eventual tumor fate of the stochastic model. This facilitates the
  deduction of Bayesian priors for a given situation, but also
  provides important insights into the underlying mechanism of tumor
  growth and progression.

  % specific results and example insights
  Although the resulting model framework is relatively simple and
  transparent, it can still reproduce the full range of known emergent
  behavior. We identify a novel model instability arising from
  nutrient starvation and we also discuss additional insight
  concerning possible model additions and the effects of those. Thanks
  to the framework's flexibility, such additions can be readily
  included whenever the relevant data become available.



% 60-XX 	Probability theory and stochastic processes
% 60Jxx 	Markov processes
% 60J22    Computational methods in Markov chains
% 60J27    Continuous-time Markov processes on discrete state spaces
% 60J28    Applications of continuous-time Markov processes on discrete 
%              state spaces
%
% 65-XX 	Numerical analysis
% 65Cxx 	Probabilistic methods, simulation and stochastic differential 
%              equations
% 65C40   Computational Markov chains
%
% 35-XX   Partial differential equations
% 35Bxx   Qualitative properties of solutions to partial differential equations
% 35B35  	Stability in context of PDEs
% 35B36  	Pattern formations in context of PDEs
%
% 92-XX 	Biology and other natural sciences
% 92-08   Computational methods
%
% 92Cxx	Physiological, cellular and medical topics
% 92C10  Biomechanics
% 92C15  Developmental biology, pattern formation

  \medskip
  \noindent
  \textbf{Statements and Declarations:} This work was partially funded
  by support from the Swedish Research Council under project number VR
  2019-03471. The authors declare no competing interests.

\end{abstract}

%**************************************************************************
\section{Introduction}
\label{sec:intro}

% (1) intro to tumors: complexity from emergent, multiscale,
% stochastic/heterogeneous
Tumors are highly complicated biological systems, yet constitute a
concrete example of cellular self-organization processes amenable to
modeling \textit{in silico} \cite{bru2003universal}. In a cancerous
tumor, the cells have undergone several significant mutations and
obtained distinct \textit{hallmarks} providing the population
with remarkable growth capabilities \cite{hanahan2000hallmarks,
  hanahan2011hallmarks}. Furthermore, populations comprising large
numbers of cells interact on multiple scales, yielding a range of
emergent phenomena \cite{deisboeck2009collective}, which can be
studied using computational models based on knowledge of single-cell
behavior. To the modeler's aid in this regard, biological data streams
nowadays contain detailed features at the individual cell level such
as cell size and -type, mutation- and growth rate, molecular
constituents, and gene expression \cite{saadatpour2015single,
  cermak2016high, armbrecht2017recent}.

% (2) from biological understanding to computational models - the
% broader purpose of
Complementing biological experiments, mathematical models can in
addition provide explanations to observed data, concerning, e.g.,
drug-response in tumor growth, with potential applications in
precision medicine \cite{yin2019review, barbolosi2016computational}.
Progress in cell biology has led to a good understanding of
intra-cellular processes which unlocks the possibility to model these
systems from fundamental principles, to translate `word models'
formulated from biological experiments into mathematical and
computational models, and to test the features of these models
\cite{roose2007mathematical}. Often quoted uses of computational
models include the testing of hypotheses, the investigation of
causality, and the integration of knowledge when comparing \textit{in
  vitro} and \textit{in vivo} data
\cite{brodland2015computational}. Bayesian inference methods present a
means to quantitatively investigate these matters, provided there
exist appropriate data and meaningful priors associated with the model
parameters.

% From http://dx.doi.org/10.1016/j.semcdb.2015.07.001
% (brodland2015computational)
%
% Test hypotheses
% Lead to new insights
% Force us to think more deeply
% Suggest and refine experiments
% Interpret experiments
% Trace chains of causation
% Carry out sensitivity analyses
% Investigate coupling and feedback
% Integrate knowledge
% Lead to new approaches

% (3) pre-existing in silico tumor models - short on pros/cons
Several cell population models exist in the literature, ranging from
continuous to agent-based to hybrid models, and taking place at
various scales \review{\cite{schluter2015multi, 
  fletcher2013implementing, szabo2013cellular, deisboeck2011multiscale, cristini2003nonlinear}}. Such models may reach
a predictive power, where agreement/disagreement with biological data
can advance our understanding of mechanistic relations within the
biological systems \cite{jin2016reproducibility,
  frieboes2009prediction, bearer2009multiparameter, drasdo2005single}.
Pertinent to the present work, previous research shows how analyzing
the emergent morphology of cell population models can provide insight
into the role of the model parameters \review{\cite{giverso2016emerging, gerlee2007stability, anderson2006tumor}},
promoting future use of Bayesian methods. Such analysis has, for
example, enabled modelers to analyze the behavior of the invasive
fronts of tumor models and their response to parameter changes
representing vascularization, nutrient availability
\cite{cristini2005morphologic, lowengrub2009nonlinear}, and cell-cell
adhesion effects \cite{byrne1996modelling, scianna2012hybrid}.

% (4) this contribution
Motivated in part by improvements of \textit{in vitro} techniques for
obtaining detailed time-series tumor and single-cell data and the
current trend in computational science towards data-driven modeling,
we present and analyze a basic continuous mathematical model of
avascular tumor growth, here derived from a previously developed
cell-based model \cite{engblom2018scalable}. Our aim is that the
model should be highly parsimonious in order to cope with issues of
model identifiability. For this purpose the \textit{in silico} tumor's
fate should be well understood when regarded as a map from parameters
to simulation end-result. Initial results from an earlier version of
the model \cite{engblom2018scalable} display boundary instabilities,
akin to those discussed in \cite{greenspan_growth_1976}, which we
analyze thoroughly. The self-regulating properties of avascular tumors
concerning size that have been observed \textit{in vitro}
\cite{folkman1973self} and \textit{in silico} \cite{grimes2016role}
motivate a careful investigation into the model capabilities in this
regard.

% (5) paper structure
We have structured the paper as follows. In \S\ref{sec:model} we
summarize the stochastic cell-based tumor model as well as the
associated mean-field space-continuous version. We analyze the latter
in \S\ref{sec:analysis}, assuming radially symmetric solutions first,
and then via linear stability analysis. In \S\ref{sec:results} we
\review{investigate via numerical examples the key aspects of} the
analysis as well as its relevance for the stochastic model. A
concluding discussion is found in \S\ref{sec:discussion}.

%**************************************************************************
\section{Stochastic modeling of avascular tumors}
\label{sec:model}

An advantage held by stochastic models is that they implicitly define
a consistent likelihood and thus formally have the potential to be
employed in Bayesian modeling when confronted with data. We summarize
our stochastic framework in \S\ref{subsec:DLCM} and the basic
stochastic tumor model in \S\ref{subsec:DLCM_model}. We next derive a
corresponding mean-field version in \S\ref{subsec:PDE_model}, which
has the distinct advantage of being open to mathematical analyses.

\subsection{Stochastic framework}
\label{subsec:DLCM}

The Darcy Law Cell Mechanics (DLCM) framework
\cite{engblom2018scalable} is a cell-based stochastic modeling
framework where the cells are explicitly represented and the rates of
their state updates, e.g., movement, proliferation, death, etc., are
determined and govern the corresponding events in a continuous-time
Markov chain. Movements of the cells generally follow \emph{Darcy's
  law} for fluid flow in a porous environment, but since the framework takes
place in continuous time, other types of cell transport are easily
incorporated.

% discrete space, carrying capacity
The spatial domain is discretized into $i = 1,2,..., \Nvox$ voxels
$v_i$, and populated by a total of $\Ncells$ cells. The DLCM framework
can be used over any grid for which a consistent discrete Laplace
operator can be derived. Each voxel may be empty or contain some
number of cells and if this number exceeds the voxel's \emph{carrying
  capacity}, the cells will exert a pressure onto the cells in the
surrounding voxels, see \figref{DLCM_pressure}. The pressure
propagates through the considered domain and the local pressure
gradient induces a cell flow. The simplest implementation allows each
voxel to be populated by $u_i \in \{0, 1, 2\}$ cells at any time, thus
with a carrying capacity of 1. \review{Note that carrying capacity
  here does not refer to the maximum possible number of cells in a
  voxel, but rather to the capacity beyond which a voxel is no longer
  in a mechanically relaxed state.}

\begin{figure}
  \centering
  \begin{subfigure}[b]{.45\textwidth}
    \centering
    \includegraphics[width=0.75\textwidth]{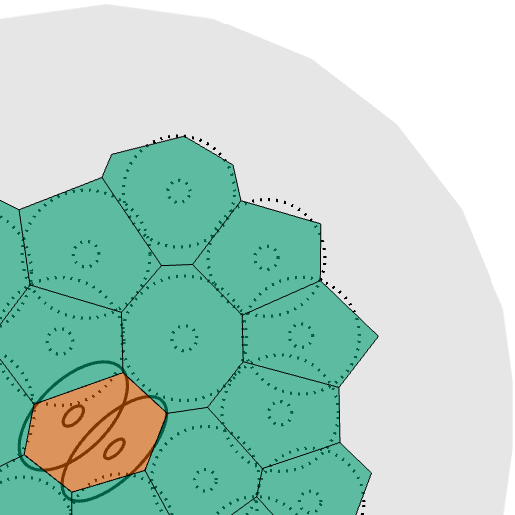}
    \subcaption{\label{fig:DLCM_pressure_a}}
  \end{subfigure}
  \begin{subfigure}[b]{.45\textwidth}
    \centering
    \includegraphics[width=0.75\textwidth]{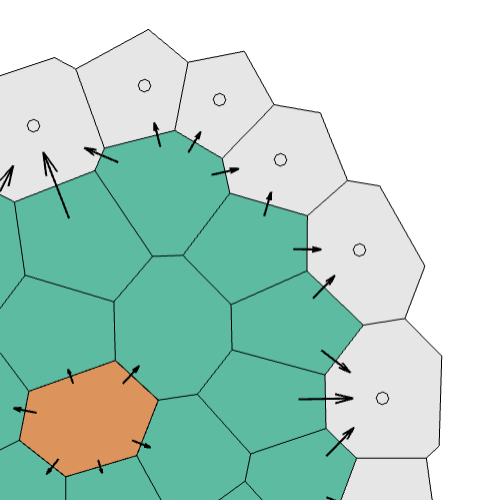}
    \subcaption{\label{fig:DLCM_pressure_b}}
  \end{subfigure}
  \caption{Cell population representation in the DLCM framework using
    a Voronoi tessellation. Green voxels represent singly occupied
    voxels and red voxels represent doubly occupied voxels as shown
    explicitly in \figref{DLCM_pressure_a}, \review{where the ellipses
      indicate the corresponding underlying population of cells. The
      grey area highlights a region of empty voxels.} The doubly
    occupied voxels exceed the carrying capacity and exert pressure on
    the surrounding cells. The movement rates are visualized in
    \figref{DLCM_pressure_b}, where the arrows represent movement rate
    and direction for a subset of the possible movements. \review{The
      grey voxels represent empty voxels that cells may migrate into.}
    Adapted from \cite[Fig.~2.1]{engblom2018scalable}.}
  \label{fig:DLCM_pressure}
\end{figure}

% (0) continuity equation, three assumptions
Although the state $u_i$ takes on discrete values in each voxel and at
any given time, the governing model is derived from a continuous
assumptions where the corresponding cell density is then $u =
u(x,t)$. We let $u$ be governed by the continuity equation
\begin{equation}
    \frac{\partial u}{\partial t} + \nabla \cdot I = 0,
    \label{eq:cont}
\end{equation}
\review{where $I$ is the flux of $u$}. There are three main
assumptions in the DLCM framework, with the first assumption pointing
to the central role of the flux $I$:
\begin{assumption}
  \label{ass:DLCM}
  Consider the discrete tissue formed from the population of cells
  distributed over the grid. We assume that:
  \begin{enumerate}
  \item\label{it:1} The tissue is in mechanical equilibrium when all
    cells are placed in a voxel of their own.
  \item\label{it:2} The cellular pressure of the tissue relaxes
    rapidly to equilibrium in comparison with any other mechanical
    processes of the system.
  \item\label{it:3} The cells in a voxel occupied by $n$ cells may
    only move into a neighboring voxel if it is occupied by less
    than $n$ cells.
  \end{enumerate}
\end{assumption}

% use these to refer to single assumptions:
\newcommand{\Assone}{Assumption~\ref{ass:DLCM}(\ref{it:1})}
\newcommand{\Asstwo}{Assumption~\ref{ass:DLCM}(\ref{it:2})}
\newcommand{\Assthree}{Assumption~\ref{ass:DLCM}(\ref{it:3})}

% (1) Darcy's law
From \Assone\ it is clear that random (e.g., Brownian) motion around a
voxel center is ignored, and, therefore, that only voxels with cells
above the carrying capacity are considered pressure
sources. \review{As in \cite{engblom2018scalable}, the flux} is determined from
the pressure gradient in the form of Darcy's law which can be derived as a
limit for flow through porous media \cite{whitaker1986flow}:
\review{
\begin{equation}
  I = -D \nabla p,
  \label{eq:flux}
\end{equation}
}
where $p$ is the pressure and the Darcy constant $D$ can be interpreted as the
ratio of the medium permeability $\kappa $ to its dynamic viscosity
$\mu$, $D := \kappa / \mu$.

% (2) pressure from solving the Laplacian
The relation between pressure and cell population is completed by
\review{assuming a constitutive relation in the form of a heat
  equation for the pressure and using \Asstwo\ to arrive at the
  stationary relation}
\begin{equation}
  -\Delta p = s(u),
  \label{eq:p}
\end{equation}
where $s(u)$ is the pressure source \review{which equals to one for
  voxels above the carrying capacity, $u_i > 1$, and zero otherwise.}

% (3) rates
Finally, we detail the flux parameter $D$ in \eqref{eq:flux}. Let
$R(e)$ denote the rate of an event $e$, and let $i \rightarrow j$
denote the event that a cell moves from voxel $v_i$ to $v_j$. With
unit carrying capacity only two distinct movements rates are possible
according to \Assthree: one for cells moving into an empty voxel, and
one for cells moving into an already occupied voxel,
\review{
\begin{align}
  \label{eq:rates}
  \left. \begin{array}{rcl}
           R(i \rightarrow j; \quad u_i \geq 1, \quad u_j = 0) &=& D_1
                                                                   I(i \to j)\\
           R(i \rightarrow j; \quad u_i > 1, \quad u_j = 1) &=& D_2
                                                                I(i \to j)
         \end{array} \right\}
\end{align}
} where $D_1$ and $D_2$ are (possibly equal) conversion factors from
units of pressure gradient to movement rate for the respective
case. Here, \review{$I(i \to j)$} is the pressure gradient integrated
over the boundary shared between the voxels $v_i$ and
$v_j$. \review{To enable a comparision of the DLCM model with a
  corresponding PDE model the effect of surface tension needs to be
  included. As a consequence of this, we also need to include
  migration between neighboring singly occupied voxels on boundaries
  where surface tension is added, and then with a rate coefficient
  equal to $D_1$. This is an event which is formally not allowed in
  the original framework, cf.~\cite{engblom2018scalable}.}

% (4) in summary...
To sum up, a population of cells occupying a grid may be evolved in
time by first solving for the pressure in \eqref{eq:p} using the
Laplacian on the grid, and then converting the pressure gradient into
rates via \eqref{eq:flux} and \eqref{eq:rates}. The rates are now
interpreted as competing Poissonian events for the corresponding
cellular movements to be simulated as a continuous-time Markov
chain. Any other dynamics taking place in continuous time are thus
readily incorporated in a consistent way.

\subsection{\review{Framework tumor model}}
\label{subsec:DLCM_model}

% the setting
As a candidate for a `minimal' avascular tumor model we consider the
one presented in \cite{engblom2018scalable} which consists of a single
cancerous cell type in three different states: \textit{proliferating},
\textit{quiescent} (i.e., dormant), and \review{\textit{necrotic}
  cells}. As a matter of convenient implementation the range of $u_i$
can be extended to include $u_i = -1$ which represents a voxel
containing a dead \review{necrotic} cell so that
$u_i \in \{ -1, 0, 1, 2\}$. Also, let $\Omega$ denote the tumor domain
with $u_i \neq 0$ and let $\Omegaext$ denote the entire computational
domain.

% details
An avascular tumor has to rely on oxygen and nutrients to diffuse
through the surrounding tissue to reach the tumor, \review{a process
assumed to be much faster than cell migration, growth, and death. As
such, nutrients are} readily modeled by \review{a stationary heat
equation} with a boundary condition on the external boundary
$\partial\Omegaext$ (far away from the tumor boundary $\partial
\Omega$) as
\begin{align}
  \left. \begin{array}{rcl}
            -\Delta c &=& - \lambda a(u) \\
           c &=& \cout \quad \text{ on } \partial \Omegaext
         \end{array} \right\}
         \label{eq:DLCM_oxygen_law}
\end{align}
where $c$ is the concentration variable understood as a proxy variable
for oxygen and any other nutrients required for the cellular
metabolism. Further, $\lambda$ is the \review{ratio of the oxygen
  consumption rate to the oxygen diffusion rate, and} $a(u_i)$ is the
number of living cells in the voxel $i$, i.e., $a(u_i) =
\max(u_i,0)$. The rates describing the tumor growth are then defined
as follows: cells \review{are in the proliferating state if
  $c_i \geq \kappaprol$ and then divide at rate $\muprol$}, where
$\kappaprol$ is the minimum oxygen concentration required for cell
proliferation. A cell dies \review{and then becomes necrotic} at rate
$\mudeath$ if $c_i < \kappadeath$, where $\kappadeath$ is the minimum
oxygen concentration required for individual cell survival. Finally,
\review{necrotic} cells degrade at rate $\mudeg$ to free up the voxel
they are in. \review{Cells in voxels at intermediate oxygen levels are
  in the quiescent state.} Note that cells \review{instantly} switch
between all living states provided the oxygen concentration allows for
it.

% surface tension and Young-Laplace BCs
At the tumor boundary a pressure condition needs to be imposed in
order to capture the net effect of cell-cell adhesion as well as the
interactions between cancerous and healthy cells. We let the
phenomenological constant $\sigma$ represent this via a Young-Laplace
pressure drop proportional to the boundary curvature $C$. Denoting by
$\pext$ the ambient pressure outside the tumor (i.e., in
$\Omegaext \setminus \Omega $) we thus have the Dirichlet condition
\begin{equation}
  p = \pext - \sigma C, \qquad \text{ on } \partial\Omega.
  \label{eq:BC_Young-Laplace}
\end{equation}
\review{An alternative approach to \eqref{eq:BC_Young-Laplace} for
  including surface tension effects directly on the microscopic level
  can be found in \cite{engblom2018scalable}, where the local adhesive
  forces work passively to resist cell migration directly by negative
  contributions to \eqref{eq:rates}. Local implementations of adhesion
  that do not change the population pressure field as with
  \eqref{eq:BC_Young-Laplace} are unfortunately not fully consistent
  with a pressure-driven migration law, thus obscuring a mechanistic
  understanding. The implementation of \eqref{eq:BC_Young-Laplace}
  along with other modifications to the DLCM framework are further
  discussed in \S\ref{apx:numerical_methods}.}

% tumor summary
A summary of the parameters of the proposed model is found in
\tabref{model_parameters}.

\begin{table}
\centering
\begin{tabular}{p{3cm}p{10cm}}
  \hline
  Parameter & Description \\
  \hline
  $D$ & Ratio medium permeability to dynamic viscosity
      $[f^{-1}l^{3}t^{-1}]$ \\
  $\lambda$ & \review{Ratio of oxygen consumption to diffusion rate $[l^{-2}]$}\\
  $\cout$ & Oxygen concentration at oxygen source $[l^{-2}]$\\
  $\muprol$ & Rate of cell proliferation $[t^{-1}]$ \\
  $\mudeath$ & Rate of cell death $[t^{-1}]$ \\
  $\mudeg$ & Rate of dead cell degradation $[t^{-1}]$\\
  $\kappaprol$ & Oxygen concentration threshold for cell proliferation \review{$[l^{-2}]$} \\
  $\kappadeath$ & Oxygen concentration threshold for cell death \review{$[l^{-2}]$} \\
  $\sigma$ & Surface tension coefficient $[f]$ \\
  \hline
\end{tabular}
\caption{Parameters of the cell-based tumor model in two
  dimensions. The same parameters are used in the corresponding PDE
  model (except for $\mudeg$ which is not used) and are nonnegative
  for both models. The units $t$, $l$, $f$ correspond to units of
  time, length, and force, respectively.}
 \label{tab:model_parameters}
\end{table}

\subsection{Mean-field PDE tumor model}
\label{subsec:PDE_model}

% introduce and motivate PDE model of tumor growth
We next derive the corresponding `minimal' partial differential
equations (PDE) model, constructed to very closely mimic the
mean-field of the stochastic model. \review{Certain aspects of the
  DLCM model, e.g., the exclusion principle in \eqref{eq:rates}, which
  excludes migration events between neighboring singly occupied
  voxels, but also \eqref{eq:DLCM_oxygen_law}, where a nonlinear
  interaction between nutrients and the cell population takes place,
  invalidate the assumption of independence between the states of
  different voxels, cf.~\cite{davies2014derivation}. This complicates
  formulating a mean-field PDE exactly. However, the underlying
  continuous physics of the DLCM model provides an appropriate
  starting point for the construction of a corresponding PDE model.}
The derivation is essentially based on mass balance with a
\review{cellular} growth- and death rates, and a velocity field
proportional to the pressure gradient.

% basic fluid mechanics
Seeing as cells comprise mostly water we assume that they are
incompressible. \review{This implies that the material derivative is
  zero, and hence the conservation law governing the tumor cell
  density $\rho$ in a velocity field $\boldsymbol{v}$ becomes}
\begin{equation}
  \frac{\partial \rho}{\partial t} + \nabla \cdot (\boldsymbol{v}
  \rho) =
  \underbrace{\frac{\partial \rho}{\partial t} + \boldsymbol{v} \cdot
    \nabla \rho}_{= 0 \text{ by incompressibility}} +
  \rho \nabla \cdot \boldsymbol{v}
  = \Gamma,
  \label{eq:PDE_conservation_law}
\end{equation}
where $\Gamma$ is cell growth and loss due to proliferation or death,
respectively, and remains to be defined. Assuming that the cells move
as a viscous fluid with low Reynolds number through a porous medium, 
the velocity field for the cell density is governed by
Darcy's law \cite{whitaker1986flow}
\begin{equation}
  \boldsymbol{v} = - D \nabla p,
  \label{eq:PDE_darcys}
\end{equation}
where the porous media that the cells reside in is the extra-cellular
matrix (ECM), whose permeability is one of two factors determining
$D$. Combining \eqref{eq:PDE_conservation_law} and
\eqref{eq:PDE_darcys} and assuming a homogeneous permeability within
the tissue domain we arrive at
\begin{equation}
  - \rho \Delta p = \frac{\Gamma}{D}.
  \label{eq:PDE_pressure_law}
\end{equation}

% source: Gamma
For the source term $\Gamma$ we mimic the stochastic model and let the
\review{cellular} growth and death rates be constant and determined by oxygen
thresholds. We thus define $\Gamma$ as
\begin{align}
  \Gamma =
  \begin{cases} \muprol \times \rho & c \geq \kappaprol \\
    0 & \kappadeath \leq c < \kappaprol \\ 
    -\mudeath \times \rho & c < \kappadeath,
  \end{cases}
  \label{eq:PDE_source_term}
\end{align}
which defines the proliferative, quiescent, and necrotic region,
respectively (see \figref{tumor_region_schematic}).  These relations
close the pressure relation \eqref{eq:PDE_pressure_law}.

% outer boundary
Similar to the cell-based formulation, we assume a Young-Laplace
pressure drop at the outer tumor boundary, which obeys
\eqref{eq:BC_Young-Laplace}. This surface tension effect arises from
various cell-cell adhesion effects --- the loss of which, due to loss
of E-cadherin function in cells, is associated with tumor metastasis
and tissue invasion \cite{hanahan2000hallmarks}. Let $\Omega$ here
denote the region with $\rho > 0$ akin to the tumor domain of the
stochastic model. By Darcy's law, we then have \review{that
  $\partial\Omega^-$ moves with the velocity
\begin{equation}
  \boldsymbol{v}_{\partial\Omega^-} = (-D\nabla p) \cdot \boldsymbol{n},
  \label{eq:PDE_compatability_condition}
\end{equation}
}
where $\boldsymbol{n}$ is the interface normal vector and
$\partial\Omega^-(t)$ denotes the boundary as approached from inside
$\Omega$. The condition \eqref{eq:PDE_compatability_condition}
connects the velocity field with the movement of the tumor boundary.

\sidecaptionvpos{figure}{c}
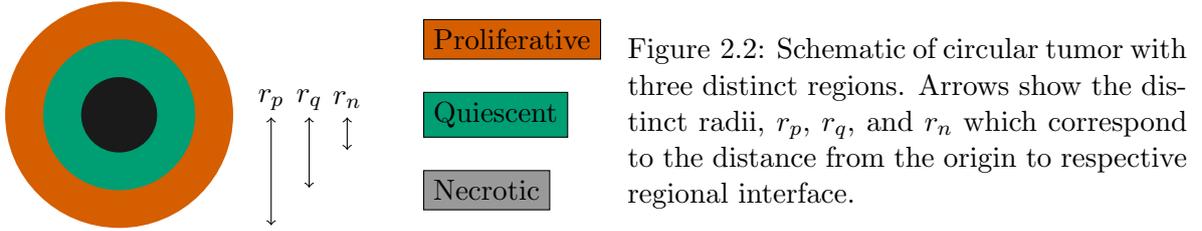
\begin{SCfigure}
  %\centering
  \begin{tikzpicture}
    % Draw the outer proliferative region
    \definecolor{vermillion}{RGB}{213,94,0}
    \fill[vermillion] (0,0) circle (1.5cm);
    % Draw the intermediate quiescent region
    \definecolor{bluish_green}{RGB}{0,158,115}
    \fill[bluish_green] (0,0) circle (1cm);
    % Draw the inner necrotic region
    \fill[black!90] (0,0) circle (0.5cm);
    % Show radii
    \node (A0) at (2, 0.2) {$r_p$};
    \node (A) at (2, 0.1) {};
    \node (B) at (2, -1.6) {};
    \node (C0) at (2.5, 0.2) {$r_q$};
    \node (C) at (2.5, 0.1) {};
    \node (D) at (2.5, -1.1) {};
    \node (E0) at (3, 0.2) {$r_n$};
    \node (E) at (3, 0.1) {};
    \node (F) at (3, -0.6) {};
    \draw[<->]
  (A) -- (B);
      \draw[<->]
  (C) -- (D);
        \draw[<->]
  (E) -- (F);
    % Add the legend
    \node[draw,rectangle,fill=vermillion,right=2cm] at (2,1) {Proliferative};
    \node[draw,rectangle,fill=bluish_green,right=2cm] at (2,0) {Quiescent};
    \node[draw,rectangle,fill=black!40,right=2cm] at (2,-1) {Necrotic};
  \end{tikzpicture}
  \caption{Schematic of circular tumor with three distinct
    regions. Arrows show the distinct radii, $r_p$, $r_q$, and $r_n$
    which correspond to the distance from the origin to respective
    regional interface.}
  \label{fig:tumor_region_schematic}
\end{SCfigure}

% include external tissue properties as tumor boundary conditions
The pressure boundary conditions concern the external medium that the
tumor grows within. Recall that $\Omegaext$ is the domain containing
both $\Omega$ and the external medium. By assuming that the external
medium obeys laws similar to the tumor tissue, we can summarize the
medium's impact on the tumor growth through the tumor boundary
conditions. We start from the assumption that the pressure propagates
freely throughout the external medium outside the tumor and, for
consistency of the complete two-tissue system, the external medium is
assumed to abide by the same assumptions as the tumor tissue (i.e.,
incompressibility, Darcy's law, and conservation of mass). We further
assume that growth and death of the external tissue is negligible and,
thus, arrive at the governing equation for the outer pressure, $\pext$
in $\Omegaext \setminus \Omega$:
\begin{eqnarray}
    \Delta \pext(\boldsymbol{r}) &=& 0,
   \label{eq:outer_pressure_finite}
\end{eqnarray}
where the solution is undetermined up to a specified boundary
condition. The velocities of the tumor and the external tissue are
both determined by Darcy's law, but are allowed different Darcy
coefficients, $D$ and $\Dout$, respectively. The velocities are
assumed compatible at the interface and the complete set of boundary
conditions on $\partial \Omega$ thus reads as
\begin{align}
    p &= \pext - \sigma C,
    \label{eq:BC_dirichlet_general} \\
     \boldsymbol{n} \cdot D\nabla p &= \boldsymbol{n} \cdot
     \Dout\nabla \pext.
     \label{eq:BC_compatability}
\end{align}
Thus, we assume
that the physical extent of the region between the tissues (where cell
mixing might occur) is negligible in comparison with the spatial scale
of the model. The nondimensional coefficient is expressed as
$\hat{D}_{\text{ext}} = \Dout/D$, but for brevity we omit the hat in
the following analysis. The compatibility condition
\eqref{eq:BC_compatability} allows for an investigation into the
effects of varying the stiffness between the tumor and the external
tissue due to, e.g., an increase in collagen density
\cite{valero2018combined}. For simplicity, we assume that the ECM
permeability is homogeneous and time-independent across the domain of
tissues, hence implicitly assuming that breakdown and remodeling of
ECM that could affect tumor progression
\cite{quail2013microenvironmental} are negligible.

% oxygen dynamics
Finally, oxygen diffuses in towards the tumor through the surrounding
tissue from a source (e.g., a vessel) far away with regards to the
spatial scale of the system. Akin to the cell-based model, \review{we
  consider a stationary heat equation but with a different source term
  as}
\begin{equation}
   - \Delta c = 
  \begin{cases}
  -\lambda \rho, \quad &c \geq \kappadeath \\ 
  0, \quad &c < \kappadeath \text{ on } \Omegaext
  \end{cases}\\
  \label{eq:PDE_oxygen_law}
\end{equation}
with $c = \cout$ on $\partial \Omegaext$, and where $\lambda$ is
\review{ratio of the} consumption rate per cell \emph{density}
\review{to oxygen diffusion rate} in the PDE setting. We assume that
$\partial \Omegaext$ is radially symmetric and lies at a distance $R$
from the domain origin and we also make the simplifying assumption
that the external tissue consumes negligible oxygen.

% nondimensionalization
A suitable choice for the characteristic length is $l_c = R$ as it
remains constant during growth. Nondimensionalization of the model
assuming a radially symmetric tumor then yields the characteristic
units
\begin{eqnarray}
  l_c = R, \quad v_c = \muprol R , \quad t_c =
  1/\muprol, \quad p_c = \muprol R^2/D, 
    \label{eq:characteristic_units}
\end{eqnarray}
with the dimensionless parameters
$\hat{\sigma} = D/(\muprol R^3) \times \sigma$ and
$\hat{\mu}_{\text{death}} = \mudeath / \muprol$. We nondimensionalize
the oxygen parameters independently of the pressure and arrive at the
characteristic units and additional dimensionless parameter,
respectively, as
\begin{align}
  \left.
  \begin{array}{l}
    \quad c_c = \cout, \quad \hat{\lambda} = \review{\lambda/\cout}, \\
    \hat{\kappa}_{\text{prol}} = \kappaprol / \cout, \quad
    \hat{\kappa}_{\text{death}} = \kappadeath / \cout.
  \end{array} \right\}
\end{align}
Subsequently, we use the characteristic units and nondimensional
parameters, but we drop the hats. The units are set to one such that
$R = 1$, $\muprol = 1$, $\cout = 1$, and $D = 1$.

% PDE model as an 'effective' model of DLCM
While the PDE model is intended to closely match the mean-field of the
DLCM model, \review{there exist} certain differences between the two
and \review{we} view the PDE model as an \textit{effective} model of
the DLCM model. For the simulation of the PDE model, we thus use
effective parameters that we derive from the outcome of the DLCM
simulations (details in \S\ref{sec:results}).

%**************************************************************************
\section{Analysis}
\label{sec:analysis}

% intro radial and morphological stability analysis
We analyze the morphological properties of the tumor model in two
spatial dimensions. This simplifies the analysis while \review{still
  allowing for a qualitative comparison with \textit{in vitro} data
  from tumor spheroids embedded in matrigel, where nutrients enter
  through the spherical surface (corresponding to the circular
  boundary in 2D)}. The case of a radially symmetric growth is
discussed in \S\ref{subsec:analysis_radial_symmetry} and a spatial
linear stability analysis in
\S\ref{subsec:linear_stability_analysis}. In essence, the outcome of
the latter analysis include conditions for when the former radially
symmetric case is a valid ansatz. Finally, in
\S\ref{subsec:analysis_stability_relations} we uncover how the
morphological instabilities of the model develop during its different
growth phases and a possible role of the external medium in
exacerbating or reducing such effects.

\subsection{Radial symmetry and the stationary state}
\label{subsec:analysis_radial_symmetry}

% the analytical expressions that relate the regional sizes given the
% oxygen parameters, and the insight they give into the 'radial'
% stability, i.e., the finite growth problem

% purpose: study radial symmetry solutions
A characteristic property of the avascular tumor is that it reaches a
stationary growth phase due to limited oxygen/nutrition availability
\cite{folkman1973self}. We therefore first derive analytical relations
that provide insight into which regions of the model parameter space
map to such a stationary state under the preliminary assumption that
the tumor is radially symmetric.

% introduce characteristic regions
The constant \review{cellular growth and death rates} define distinct
characteristic regions of the model tumor according to
\eqref{eq:PDE_source_term}: the proliferative, quiescent, and necrotic
region. For a radially symmetric tumor at time $t \ge 0$, motivated by
the form of the oxygen field governed by \eqref{eq:PDE_oxygen_law}, we
let $r_p(t)$ denote the tumor radius, $r_q(t)$ the radius of the
interface between the proliferative and quiescent region, and $r_n(t)$
the radius of the interface between the quiescent and necrotic region,
cf.~\figref{tumor_region_schematic}. Given $\lambda > 0$, the assumed
radial symmetry implies that $0 \le r_n \le r_q \le r_p$, and under
the chosen units, $r_p < R = 1$. We note that if $r_p$ is sufficiently
close to the oxygen source at $R$, the model assumption of avascularity
breaks down.

% additional assumptions before analysis
We simplify the problem by assuming that $\rho = 1$ across the entire
tumor domain since this allows the oxygen field to be explicitly
solved. By incompressibility and slow migration of cells, this is a
reasonable approximation and, besides, a constant cell density in the
PDE model is a close match to the discrete stochastic model that we
wish to investigate. 

% analysis details and resulting relations
We thus solve \eqref{eq:PDE_oxygen_law} under radial symmetry while
imposing $C^1$-continuity for the oxygen $c$ across the interfaces
between the characteristic regions. Under radial symmetry, the
divergence theorem \review{applied to \eqref{eq:PDE_oxygen_law}}
implies an inhomogeneous Neumann boundary condition across each radial
boundary whose value is proportional to the volume of oxygen sinks
contained within. \review{The full problem for the oxygen under radial
  symmetry thus reads
\begin{equation}
  \begin{aligned}
    -\frac{\partial }{r \partial r} \bigg( \frac{r \partial c }{ \partial r}\bigg) = s_c(r) := \begin{cases}
      -\lambda, \quad  &\interface{n} \leq r \leq \interface{p}, \\ 
      0, \quad  &\text{otherwise}
    \end{cases} \\
    c(1) = 1, \quad
    c'(r_i) =
    -\frac{1}{r_i}\int_0^{r_i} s_c(s)s \, ds, \;
    \mbox{for } r_i \in \{\interface{n}, \interface{q}, \interface{p}\},
    \label{eq:PDE_oxygen_law_radialsymmetry}
  \end{aligned}
\end{equation}
where the last three relations follow from application of the
divergence theorem under radial symmetry and at each interface
separately. We solve \eqref{eq:PDE_oxygen_law_radialsymmetry} and
find}
\begin{align}
  \label{eq:PDE_oxygen_radialsymmetry}
  c(r) &= \left\{ \begin{array}{ll}
                    1 + \lambda/2 \, \left( r_p^2- r_n^2 \right) \log r, &
                    r_p \leq r,\\
                    1 + \lambda/2 \, \left( r_p^2 \log r_p
                    -r_n^2 \log r  + (r^2 - r_p^2)/2 \right), & r_n \leq r
                    \leq r_p, \\
                    \kappadeath, & r \leq r_n,
                  \end{array} \right.
\end{align}
where the expression does not differ in the proliferative and
quiescent regions since the oxygen consumption rates are identical there. By
definition, the oxygen level at $r_q$ is $\kappaprol$ and at $r_n$ it
is $\kappadeath$. Using this, \eqref{eq:PDE_oxygen_radialsymmetry}
implies the following algebraic relations between the characteristic
regions
\begin{align}
  \label{eq:PDE_regional_relations}
  \left. \begin{array}{rclcl}
           \Kprol &:=& 4(1-\kappaprol)/\lambda &=&
                                                   -r_p^{2} \log r_p^2+
                                                   r_n^2 \log r_q^2 -
                                                   r_q^2 + r_p^{2} \\
           \Kdeath  &:=& 4(1-\kappadeath)/\lambda &=&
                                                      -r_p^{2} \log r_p^2 +
                                                      r_n^2 \log r_n^2
                                                      -r_n^2 + r_p^{2}
         \end{array} \right\}                                    
\end{align}
in terms of the reduced parameter set $\{\Kprol,\Kdeath\}$.

% DAE dynamics
Given radial symmetry and $\rho = 1$, the tumor volume change is
derived from mass balance as $\dot{V} = \muprol V_p - \mudeath V_n $,
where $V_p$ and $V_n$ are the volumes of the proliferative and
necrotic region, respectively. Thus, in two spatial dimensions we get
that
\begin{eqnarray}
  \label{eq:PDE_regional_dynamics}
  d/dt \, \left(r_p^2 \right) =  -\mudeath r_n^2 - r_q^2 + r_p^{2},
\end{eqnarray}
under the nondimensionalization where $\muprol = 1$. \review{Hence an
  initial state $r_p(t = 0)$ together with the reduced parameter set
  $\{\Kprol,\Kdeath,\mudeath\}$ fully determine the dynamics of a
  radially symmetric tumor under
  \eqref{eq:PDE_regional_relations}--\eqref{eq:PDE_regional_dynamics}.}

% linear stability analysis
Assume now that $(\rneq,\rqeq,\rpeq)$ is a stationary solution of
\eqref{eq:PDE_regional_relations}--\eqref{eq:PDE_regional_dynamics}. \review{Writing
  $\interface{q}^2 = \interface{q}^2(\interface{p}^2)$ and
  $\interface{n}^2 = \interface{n}^2(\interface{p}^2)$, and by
  implicitly differentiating \eqref{eq:PDE_regional_relations}} we can
linearize \eqref{eq:PDE_regional_dynamics} around this solution and
retrieve the single eigenvalue
\begin{align}
  \label{eq:rad_eigenvalue}
  \Lambda_r &= 1 - \frac{\log \rpeq}{\log \rneq} \left(\mudeath +
              \frac{2 \log \frac{\rqeq}{\rneq}}{(\rqeq)^2 -
              (\rneq)^2} \, (\rqeq)^2\right).
\end{align}

\begin{proposition}[\textit{Stability of radially symmetric equilibrium}]
  \label{proposition:rad_stab}
  Given $\mudeath > 0$, assume that $0 \le \rneq \le \rqeq \le \rpeq$
  is a stationary solution of
  \eqref{eq:PDE_regional_relations}--\eqref{eq:PDE_regional_dynamics}. Then
  $\Lambda_r < 0$ in \eqref{eq:rad_eigenvalue} whenever
  $\rpeq \le \exp(-1) \approx 0.368$.
\end{proposition}

\begin{proof}
  Put $(\rneq)^2 = \eta (\rqeq)^2$ for some $\eta \in (0,1)$ and note
  that $(\rpeq)^2 = (1+\mudeath \eta) (\rqeq)^2$ by stationarity (the
  cases $\eta \in \{0,1\}$ are treated as limits). The eigenvalue
  becomes
  \begin{align*}
    \Lambda_r &= 1 - \frac{\log \rpeq}{\log \rpeq+\frac{1}{2} \log
                \left( \eta/(1+\mudeath \eta) \right)} \left(\mudeath -
                \frac{\log \eta}{1-\eta} \right).
  \end{align*}
  By inspection we find that as a function of $\mudeath$, the
  expression on the right is monotonically decreasing and hence it is
  bounded by its behavior as $\mudeath \to 0+$:
  \begin{align*}
    \Lambda_r &< 1 + \frac{\log \rpeq}{\log \rpeq+\frac{1}{2} \log \eta} 
                \times \frac{\log \eta}{1-\eta}.
  \end{align*}
  In turn, as a function of $\rpeq$, this expression is monotonically
  increasing such that, in particular, for $\rpeq \le \exp(-1)$ we
  have that
  \begin{align*}
    \Lambda_r &< 1 + \frac{\log \eta}{1-\frac{1}{2} \log \eta} 
                \times \frac{1}{1-\eta}.
  \end{align*}
  From the elementary inequality $-y \le (\exp(-y)-1) \cdot (1+y/2)$
  for $y \ge 0$ we conclude, taking $y = -\log \eta$, that
  $\Lambda_r < 0$. \review{The same inequality applies also to the
    limits $\eta \rightarrow \{0^+,1^-\}$.}
\end{proof}

% comment on Proposition
Proposition~\ref{proposition:rad_stab} \review{provides a sufficient
  condition for stability based only on the size of the tumor and} can
be thought of as a modeling cut-off: either a radially symmetric tumor
is small enough to be stable, or it has grown too large relative to
the oxygen source for stability to be guaranteed \review{for all
  $\eta$. While there might exist values of $\eta$ for which a larger
  tumor is stable, the size alone for $\rpeq > \exp(-1)$ is not
  sufficient to deduce that the tumor is stable, but the converse is
  true for smaller tumors.} In the numerical experiments in
\S\ref{sec:results}, we use Proposition~\ref{proposition:rad_stab} to
ensure that the tumor is small enough that a radially symmetric
solution is expected to reach a stable stationary state. For suggested
stationary radii $(\rneq,\rqeq,\rpeq)$, with $\rpeq$ small enough,
$\mudeath$ is defined by setting \eqref{eq:PDE_regional_dynamics} to
0, and similarly $\Kprol$ and $\Kdeath$ are found from
\eqref{eq:PDE_regional_relations}\review{, which are used to determine
  the stability of the stationary state and the full dynamics of the
  tumor's growth rate.}

\subsection{Morphological stability}
\label{subsec:linear_stability_analysis}

% intro main results
We analyze the stability of the PDE model by studying the system's
response to perturbations of a radially symmetric solution.  The main
result depends on the three Lemmas in
\S\ref{apx:propositions_stability_analysis} and reads as follows:
% state the main result
\begin{theorem}[\textit{Linear stability}]
  Let the outer tumor boundary $\interface{p}$ be perturbed by
  \review{
  \begin{equation}
    \rpert{p}(\theta) = \interface{p} + \epsilon \rcoef{p}\cos(k\theta),
    \label{eq:front_perturbation}
  \end{equation}
  }
  for some $\lvert \epsilon \rvert \ll 1$. Write the induced inner
  perturbations on the same form,
  \review{
  \begin{align}
    \label{eq:inner_regions_perturbations}
    \begin{split}
      \rpert{q}(\theta) &= \interface{q} + \epsilon \rcoef{q}
      \cos(k\theta),
      \\
      \rpert{n}(\theta) &= \interface{n} + \epsilon \rcoef{n}
      \cos(k\theta),
    \end{split}
  \end{align}
} each defined as the interface between the regions of different
\review{cellular} growth rates according to \eqref{eq:PDE_source_term}
with the oxygen field governed by \eqref{eq:PDE_oxygen_law} (see also
\figref{tumor_region_schematic}). We let the pressure field and the
cell density advection be defined as in
\S\ref{subsec:PDE_model}. \review{Then to first order in $\epsilon$,
  the $k$th perturbation mode in \eqref{eq:front_perturbation} grows
  as $\rcoef{p}(t) \propto e^{\Lambda(k) t}$ according to the
  dispersion relation}
  \begin{align}
    \label{eq:dispersion_relation_explicit}
    \Lambda(k) &= \frac{\interface{p}'}{\interface{p}}\Biggl(
    \overbrace{\frac{1 - \Dout}
                 {1+\Dout}k }^{\text{Saffman-Taylor}} - 1 \Biggr) \\
    \nonumber
    + &\frac{\Dout}{1+\Dout} \Biggl(1 - \underbrace{\interface{p}^{-k
    -1}\left(\mudeath
        \interface{n}^{k+1}\frac{\rcoef{n}}{\rcoef{p}} +
        \interface{q}^{k+1}\frac{\rcoef{q}}{\rcoef{p}}\right)}_{\text{inner region perturbation}} -
        \underbrace{\sigma
        \frac{k(k^2-1)}{\interface{p}^{3}}}_{\text{surface tension}}\Biggr),
  \end{align}
  in which the interface perturbation coefficients
  $\rcoef{q}$, $\rcoef{n}$ are given by
  \eqref{eq:regional_perturbation_coefficients} and where the radial
  growth follows from \eqref{eq:PDE_regional_dynamics},
  \begin{align}
    \label{eq:PDE_front_velocity}
    \interface{p}' &= -\frac{1}{2\interface{p}}(\mudeath \interface{n}^2
    + \interface{q}^2 - \interface{p}^{2}) = -p{'}(\interface{p}),
  \end{align}
  by Darcy's law.
\end{theorem}

\begin{proof}
  % pressure field for the unperturbed case
  We first solve for the pressure field \eqref{eq:PDE_pressure_law}
  under the assumption of radial symmetry where, again, the divergence
  theorem implies Neumann interface conditions. The result is
  \begin{align}
    \label{eq:PDE_pressure_outerregion}
    \pext(r) &= (\mudeath \interface{n}^2 +
               \interface{q}^2 - \interface{p}^2) \frac{1}{2\Dout} \log(r),
               &\quad \interface{p} \leq r&, \\
    \label{eq:PDE_pressure_proliferatingregion}
    \subquant{p}{p}(r) &= (\mudeath \interface{n}^2 + \interface{q}^2)
                         \frac{1}{2} \log \frac{r}{\interface{p}} - \frac{1}{4}(r^2 -
                         \interface{p}^{2}) + \pext(\interface{p}) - \sigma C, &\quad
                         \interface{q} \leq r& \leq \interface{p}, \\
    \nonumber
    \subquant{p}{q}(r) &= \subquant{p}{p}(\interface{q}) +
                         \frac{1}{2}\interface{n}^2\log\frac{r}{\interface{q}}, &\quad \interface{n}
                         \leq r& \leq \interface{q} , \\
    \label{eq:PDE_pressure_necroticregion}
    \subquant{p}{n}(r) &= \subquant{p}{q}(\interface{n}) +
                         \frac{\mudeath}{4}(r^2 - \interface{n}^2), &\quad 0 \leq r& \leq
                         \interface{n},
  \end{align}
  where the negative pressure gradient at $\interface{p}$
  \eqref{eq:PDE_pressure_proliferatingregion} recovers the velocity of
  the tumor's outer rim \eqref{eq:PDE_regional_dynamics} as expected.
  
  % stability analysis
  \review{We continue by perturbing the tumor boundary according to
  \eqref{eq:front_perturbation} and, assuming independence between modes for the
  induced perturbations, we have \eqref{eq:inner_regions_perturbations} and
  let the pressure perturbation in each region $i \in 
  \{\text{ext},p,q,n\}$ be
  \begin{align*}
    \ppert{i}(r) &=
                   \subquant{p}{i}(r) + \epsilon \pcoef{i}(r) \cos (k
                   \theta).
  \end{align*}
  }
  Using similar arguments to those in the proof of
  Lemma~\ref{lemma:oxygen_perturbation_resonse}, we can show that the
  pressure perturbation coefficients are of the form
  \eqref{eq:oxygen_perturbation_form}. We next use
  Lemmas~\ref{lemma:continuity} and \ref{lemma:continuous_derivative}
  to find the continuity relations for the coefficients. However, the
  pressure discontinuity at the tumor boundary,
  \eqref{eq:BC_dirichlet_general} and \eqref{eq:BC_compatability},
  must be treated separately. For the former, the first order
  approximation in $\epsilon$ of $\sigma C(\interface{p}^*)$ is
  evaluated, and for the latter we use that $D$ and $\Dout$ are
  constant within their respective domains. Thus, the continuity
  relations become
  \begin{align}
  \label{eq:perturbation_continuity_relations}
    \begin{split}
      \pcoef{p}(\interface{p}) &= \pextcoef(\interface{p}) +
      (\pext{'}(\interface{p}) - p{'}(\interface{p})
      + \sigma \frac{k^2 - 1}{\interface{p}^{2}}) \rcoef{p}, \\
      \pcoef{p}(\interface{q}) &= \pcoef{q}(\interface{q}), \\
      \pcoef{q}(\interface{n}) &= \pcoef{n}(\interface{n}),
  \end{split}
  \end{align}
  and for the derivatives,
  \begin{align}
  \label{eq:perturbation_derivative-continuity_relations}
    \begin{split} \pcoef{p}{'}(\interface{p}) &= -
      p{''}(\interface{p})\rcoef{p} +
      \Dout(\pextcoef{'}(\interface{p}) +
      \pext{''}(\interface{p})\rcoef{p}), \\ \pcoef{p}{'}(\interface{q}) &=
      \pcoef{q}{'}(\interface{q}) +
      \rcoef{q}(\subquant{p}{p}{''}(\interface{q}) -
      \subquant{p}{q}{''}(\interface{q})), \\ \pcoef{q}{'}(\interface{n}) &=
      \pcoef{n}{'}(\interface{n}) +
      \rcoef{n}(\subquant{p}{q}{''}(\interface{n}) -
      \subquant{p}{n}{''}(\interface{n})). \\
    \end{split}
  \end{align}
  As in
  \cite{giverso2016emerging}, we find the form of this dispersion
  relation from the velocity at the tumor boundary by applying
  \review{\eqref{eq:PDE_compatability_condition} to the
  perturbed solution. Considering only the first order terms in
  $\epsilon$, we find that $\partial\rcoef{p} / \partial t = \Lambda(k) \rcoef{p}$, with}
  \begin{equation}
    \Lambda(k) = -\left(p{''}(\interface{p}) +
      \frac{\pcoef{p}{'}(\interface{p})}{\rcoef{p}}\right).
    \label{eq:dispersion_relation}
  \end{equation}
  Finally, evaluating \eqref{eq:dispersion_relation} using
  Lemma~\ref{lemma:oxygen_perturbation_resonse} for the coefficients
  $\rcoef{q}$, $\rcoef{n}$ yields the dispersion relation
  \eqref{eq:dispersion_relation_explicit}.
\end{proof}

Note that
\eqref{eq:dispersion_relation_explicit} is independent of the
coefficients $\rcoef{p}$ of the initiating perturbation since both
$\rcoef{n}$ and $\rcoef{q}$ are proportional to $\rcoef{p}$, as seen in
\eqref{eq:regional_perturbation_coefficients}. It follows that $\Lambda(k)$
is unambiguously determined by the set of parameters and values
$\{\interface{p}, \interface{q}, \interface{n}, \mudeath, \Dout,
\sigma\}$.

% briefly discuss key feature of the stability results
The first part in \eqref{eq:dispersion_relation_explicit} is the
Saffman-Taylor instability \cite{saffman1958penetration} term. This
type of instability has previously been discussed in the context of
growing cell populations \cite{mather2010streaming}. Further, the
induced perturbations on the oxygen field act to dampen any
morphological instability as seen by the inner region perturbation
term, which is negative and can only decrease the value of
$\Lambda(k)$. However, for large $k$ this dampening vanishes in
general as is seen from the following reasoning. Let
$\interface{n}^2 = \theta_n \interface{p}^2$,
$\interface{q}^2 = \theta_q \interface{p}^2$, with
$0 \leq \theta_n < \theta_q < 1$, i.e., the regions do not
overlap. Then the inner region perturbation term becomes
\begin{equation}
    \label{eq:oxygen_perturbation_simplified}
    \frac{1 - \interface{p}^{2k}}{1 - \theta_n^k
    \interface{p}^{2k}} \left( \mudeath \theta_n^{k} + \frac{\theta_q^{k} -
    \theta_n^{k}}{\theta_q - \theta_n} \times \frac{1}{k} \times
    \theta_q \right),
\end{equation}
which for $\interface{p} < 1$ tends to zero as $k$ grows.  Finally,
surface tension also reduces the amplitude and range of unstable
perturbation modes. Similar effects are observed due to cell adhesion
in glioblastoma models \textit{in silico} and \textit{in vitro}
\cite{oraiopoulou2018integrating}.

\subsection{Notable special cases}
\label{subsec:analysis_stability_relations}

The dispersion relation \eqref{eq:dispersion_relation_explicit}
provides rich insight into the morphological dynamics of our
model of a growing avascular tumor. We outline notable regimes of
these dynamics below.

% (1) Here we show how the Darcy coefficient relates to the ambient pressure
% Large \Dout => homogeneous ambient pressure
% Small \Dout => unstable
\paragraph{The Saffman-Taylor instability}

When the tumor grows in a medium that flows on a significantly smaller
timescale than the \review{migration rate of the tumor cells}, corresponding to $\Dout \gg D$
($\equiv 1$ by nondimensionalization), the dispersion relation becomes
\begin{align}
  \label{eq:dispersion_relation_largeDout}
  \Lambda(k) &= -\frac{\interface{p}'}{r_p}( k + 1 ) \\
  \nonumber
  &\phantom{=}+1 - r_p^{-k-1}\left(\mudeath r_n^{k+1}\frac{\rcoef{n}}{\rcoef{p}} +
    r_q^{k+1}\frac{\rcoef{q}}{\rcoef{p}}\right) - \sigma \frac{k(k^2-1)}{r_p^{3}}.
\end{align}
\review{The same result is obtained by assuming a homogeneous outer
  pressure, $\pext(r) = p_0$, for some constant $p_0$. Following the
  same arguments we get again
  \eqref{eq:perturbation_continuity_relations} but with $\pextcoef$
  and $\pext{'}$ equal to zero, and that
  \eqref{eq:perturbation_derivative-continuity_relations} holds except
  for the first relation, since $C^1$-continuity is no longer valid on
  $\partial \Omega$ making Lemma~\ref{lemma:continuous_derivative}
  inapplicable there. The dispersion relation is still given by
  \eqref{eq:dispersion_relation} and
  Lemma~\ref{lemma:oxygen_perturbation_resonse} also holds (the oxygen
  field does not explicitly depend on the outer pressure), which
  combined yields \eqref{eq:dispersion_relation_largeDout}.} The
Saffman-Taylor term (the first term) is here at its most stabilizing
since $k(1 - \Dout)/(1 + \Dout) \geq -k$ for $\Dout \geq 0$. On the
other side of the spectrum, we have the situation when the external
tissue is significantly more viscous and practically immovable within
the temporal scale of the growing tumor. This corresponds to the
condition $\Dout \ll 1$, and \eqref{eq:dispersion_relation_explicit}
becomes
\begin{eqnarray}
    \Lambda(k) &=& \frac{\interface{p}'}{r_p}( k - 1 ),
    \label{eq:dispersion_relation_largeD}
\end{eqnarray}
and every mode $k \geq 2$ is unstable during growth with no stabilizing effect from the surface
tension. The case $\Dout < 1$ is the common form of the Saffman-Taylor
instability.

% (2) Here we outline the importance of the discriminant of the
% boundary instability. This analysis reveals that these instabilities
% become stronger and permits larger and larger modes as the tumor
% growth slows down. This is an indicator that low levels of surface
% tension are more prone to metastasis and continued growth of the
% tumor.

% introducing the discriminant
\paragraph{Growth phases}

% initial growth
The tumor's morphological stability depends on which growth phase the
tumor is in. We identify the following quantity as a discriminant of
the tumor's growth phase:
\begin{equation}
    \Delta_{\theta} = \frac{\interface{p}'}{\interface{p}},
    \label{eq:angular_discriminant}
\end{equation}
i.e., the relative tumor boundary velocity. When the tumor is
initially growing in a nutrient-rich environment \review{and $r_n$,
  $r_q$ are small, then from \eqref{eq:PDE_front_velocity} we see that
  $\interface{p}' \approx \interface{p}/2$ which implies that}
$\Delta_{\theta} \approx \frac{1}{2}$ and hence
\eqref{eq:dispersion_relation_explicit} becomes
\begin{equation}
    \Lambda(k) = \frac{1-\Dout}{2(1 + \Dout)}(k-1) - \frac{\Dout}{1+\Dout} \sigma
    \frac{k(k^2-1)}{r_p^3}.
\end{equation}
Hence as the tumor grows exponentially, we expect all modes $k > 1$ to
be stable for $\Dout > 1$ and unstable for $\Dout < 1$ in the absence
of surface tension effects.

% stationary state
As the tumor grows larger and the nutrient availability can no longer
sustain the entire tumor, the tumor front slows down and
$\Delta_{\theta} \rightarrow 0^+$, and the effect from the
Saffman-Taylor part diminishes. From
\eqref{eq:oxygen_perturbation_simplified} we see that the inner region
perturbation term in \eqref{eq:dispersion_relation_explicit} is
negative, and hence close to the tumor's stationary state we have that
\begin{align}
    \label{eq:stationary_state_instability}
  \Lambda(k) < 
  \frac{\Dout}{1+\Dout} \Bigl(1 -
      \sigma
      \frac{k(k^2-1)}{\interface{p}^{3}}\Bigr),
\end{align}
Since the inner region perturbation term tends to zero for increasing
$k$, the upper bound becomes a good approximation of $\Lambda(k)$ for
large $k$.  Clearly, \eqref{eq:stationary_state_instability} shows
that a positive value of $\sigma$ is necessary for the stationary
stability.

\paragraph{Surface tension and stationarity}

The stability relation provides an estimate of the surface tension
required to maintain a radially symmetric growth as
$t \rightarrow \infty$. Considering only the case when $\Dout \gg 1$,
we see from \eqref{eq:dispersion_relation_largeDout} that the least
stable case with respect to the discriminant is obtained when
$\Delta_{\theta} = 0$. We find the necessary surface tension by
requiring $\Lambda(k) = 0$ for all $k$. Again, using that the inner
region perturbation term is negative, we obtain from
\eqref{eq:dispersion_relation_largeDout} the bound
\begin{equation}
  \sigma_{\text{stable},k} \leq \frac{r_p^{3}}{k(k^2 - 1)}
  \label{eq:surface_tension_stable_k}
\end{equation}
where $\sigma_{\text{stable},k}$ is the lower bound of $\sigma$
required for stabilizing small perturbations of mode $k$. Thus, to
stabilize all modes $k \geq 2$ it is \textit{sufficient} to have that
$\sigma = r_p^{3}/6$, depending only on the total tumor volume. Again,
since the inner perturbation term tends to zero as $k$ grows, the
upper bound is a good approximation of $\sigma_{\text{stable},k}$ for
large $k$.

% (3) uninhibited 'creeping' instability
\paragraph{Creeping instability}

We finally remark on the interesting mode $k=1$, the only mode
unaffected by surface tension. Geometrically, this mode corresponds to
movement of the tumor's center of mass: the tumor begins to
\textit{creep} towards the oxygen source given a small
perturbation. From
\eqref{eq:dispersion_relation_explicit},
\begin{equation}
  \label{eq:creeping_instability}
  \Lambda(k=1) = \frac{\Dout}{1 + \Dout}\frac{\mudeath r_n^2 + r_q^2}{r_p^2}
  \times \frac{r_p^2 - r_n^2}{1-r_n^2}  \geq 0.
\end{equation}
Thus, the tendency for creeping always exists when $\interface{q}$
and/or $\interface{n}$ are positive, and the model would require
additional features regarding, e.g., the external tissue's response to
invasion, in order to inhibit this effect. Note that this tendency is
reduced for tumors growing within an environment more viscous and/or
less permeable than its own. As showed previously, however, such
conditions make all the other modes $k\geq2$ \textit{less} stable.

%**************************************************************************
\section{Numerical examples}
\label{sec:results}

% intro
In the following section, we present numerical simulations of both the
stochastic model from \S\ref{subsec:DLCM_model} and the PDE model from
\S\ref{subsec:PDE_model}. We focus on the case $\Dout \gg 1$ which we
showed had an inherent stabilizing effect during growth in
\S\ref{subsec:analysis_stability_relations}.  In
\S\ref{subsec:results_radial_symmetry} we assess the validity of the
assumption that growth is radially symmetric and how the stability
responds to surface tension. In \S\ref{subsec:results_stochasticity},
we explore the relation between surface tension and the emergent
morphology and compare the outcomes between the stochastic and the
mean-field PDE model.

% PDE as 'effective' model
Due to certain differences between the models, we use effective
parameter values for the PDE simulations for the parameters $\muprol$,
$\mudeath$ and $\lambda$, and denote those with a bar, e.g.,
$\muprolbar$. The effective parameters are derived from the stochastic
model simulations via basic scaling considerations or preliminary
simulations \review{as detailed in \S\ref{apx:numerical_methods}}.

% parameters used
The set of parameters used in both the DLCM and the PDE simulations
are found in \tabref{PDE_parameters_numerical}.

\review{
\begin{SCtable}
  \centering
  \begin{tabular}{p{3cm}p{2cm}p{2cm}}
    \hline
    Parameter & DLCM & PDE \\
    \hline
    $\mudeath$ & 0.5 & 1.35 \\
    $\mudeg$ & 0.05 & N/A \\
    $\kappaprol$ & 0.94 & 0.94 \\
    $\kappadeath$ & 0.93 & 0.93   \\
    $\lambda$ & 1 & 1.15\\  
                         $\pext$ & 0 & 0 \\
                        $\Dout$ & $+ \infty$ & $+ \infty$ \\
   $D_2$ & 25 & N/A \\ 
    \hline
  \end{tabular}
  \caption{\review{Standard set of parameters of the DLCM model and
      the corresponding effective PDE parameters.}}
  \label{tab:PDE_parameters_numerical}
\end{SCtable}
}

\subsection{Radial symmetry and surface tension}
\label{subsec:results_radial_symmetry}

% (2a) solution to PDE assuming radial symmetry
We first solve the PDE under the assumption of radial symmetry. We
solve the reduced 1D problem comprising
\eqref{eq:PDE_regional_relations} and \eqref{eq:PDE_regional_dynamics}
as derived in \S\ref{subsec:analysis_radial_symmetry} and evaluate the
perturbation growth \review{rates} \eqref{eq:dispersion_relation_largeDout}
during tumor growth and study their response to surface tension.  For
quantitative measurements of the regional characteristics during tumor
growth, we consider the volumetric quantities $V_p = \pi r_p^2$,
$V_q = \pi r_q^2$, and $V_n = \pi r_n^2$.

% (2a) simple first example of PDE solution dynamics in 1D and
% lambda(k) over time for a 1D finite growth solution.
\begin{figure}[ht]
  \centering
  \begin{subfigure}{.5\linewidth}
  \centering
    \includegraphics[scale=1]{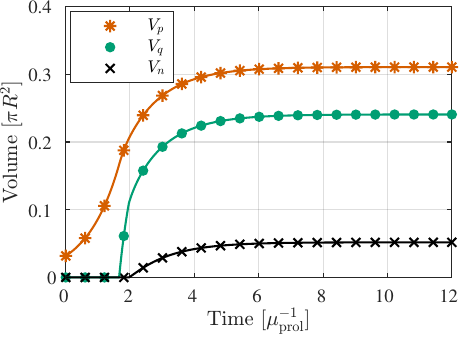}
    \caption{\label{fig:PDE_solution_example_a}}
  \end{subfigure}
  \begin{subfigure}{.47\linewidth}
  \centering
    \includegraphics[scale=1]{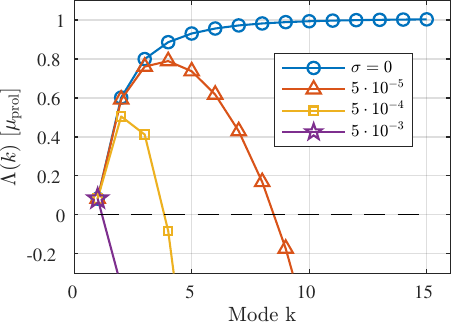}
    \caption{\label{fig:PDE_solution_example_b}}
  \end{subfigure}
  \caption{Regional characteristics and perturbation growth \review{rates} under
  the assumption of radial symmetry. \figref{PDE_solution_example_a}
  shows the evolution of the three characteristic regions.
  \figref{PDE_solution_example_b} shows the
  perturbation growth \review{rates} \eqref{eq:dispersion_relation_largeDout}
  for the first few modes at the tumor's stationary state and at varying
  surface tension $\sigma$.}
  \label{fig:PDE_solution_example}
\end{figure}
  
% observations to above figure
\figref{PDE_solution_example} shows the evolution of the regional
characteristics for a simulation using the standard parameters for the
PDE found in \tabref{PDE_parameters_numerical}, accompanied by the
perturbation growth \review{rates} \eqref{eq:dispersion_relation_largeDout}
at the stationary state. In \figref{PDE_solution_example_a} we observe
the emergence of the characteristic sigmoidal growth of the total volume, with an
initially exponential growth followed by a growth rate that plateaus.
\figref{PDE_solution_example_b} shows the perturbation growth \review{rates}
versus mode close to the stationary state for a range of $\sigma$.  It is clear
that the assumption of radial symmetry for low values of $\sigma$ does not hold
when oxygen is not sufficient to sustain the growth of the entire
population. From \eqref{eq:dispersion_relation_largeDout} we find that
$\sigma_{\text{stable},2} \approx 3.1\cdot 10^{-3}$, i.e., this is the
value needed to stabilize modes $k \geq 2$. From
\figref{PDE_solution_example_b} we see that $\sigma$ lower than this
prompt nontrivial spatial behavior with instabilities that may occur
over different timescales (investigated further in
\S\ref{subsec:results_stochasticity}).
% end Example 2a

% (2b) creeping (unstable for k = 1)
\begin{SCfigure}
  \begin{subfigure}[b]{.175\linewidth}
    \includegraphics[scale=1]{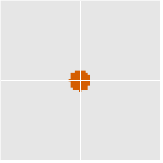}
    \caption{\label{fig:PDE_creeping_example_a}}
  \end{subfigure}
  \begin{subfigure}[b]{.175\linewidth}
    \includegraphics[scale=1]{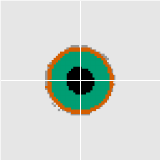}
    \caption{\label{fig:PDE_creeping_example_b}}
  \end{subfigure}
  \begin{subfigure}[b]{.175\linewidth}
    \includegraphics[scale=1]{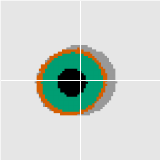}
    \caption{\label{fig:PDE_creeping_example_c}}
  \end{subfigure}
  \caption{\review{Solution using a surface tension large enough to ensure
    stability and hence radial symmetry, $\sigma = 3.2\cdot 10^{-3}$.
    Times $t=0$, $t=10$, $t=80$, respectively, from left to
    right. \textit{Red:} proliferative cells, \textit{green:}
    quiescent cells, \textit{black:} necrotic cells, \textit{dark
      grey:} previously occupied. The white cross-hairs are centered
    on the domain origin.}}
  \label{fig:PDE_creeping_example}
\end{SCfigure}

% observations to above figure
As suggested by \eqref{eq:creeping_instability}, creeping is expected
for long enough times. We test this \review{by simulating the model using
$\sigma = 3.2\cdot 10^{-3}$ to ensure that modes $k \geq 2$ are stabilized
(see details about the numerical methods and the implementation of
surface tension in \S\ref{apx:numerical_methods}). The results are shown
in \figref{PDE_creeping_example} where we see that the} tumor
reaches close to its stationary state at around \review{$t = 10$}, in
accordance with the solution to the 1D equations in
\figref{PDE_solution_example}. We observe a notable collective migration from
the domain origin starting at \review{$t = 80$}.
% end Example 2b

\subsection{The emergent morphology and its response to stochasticity}
\label{subsec:results_stochasticity}

% comparison between analytical and obeserved instability for DLCM
\review{We begin by a brief investigation into how well the dispersion
  relation \eqref{eq:dispersion_relation_largeDout} describes the
  morphological stability of the DLCM model tumor. To this end, we
  conduct simulations of the complete DLCM model, initialized close to
  the estimated equilibrium volumes (see
  \figref{PDE_solution_example_a}) using the standard parameters and
  setting $\sigma = 10^{-4}$. We impose an initial perturbation to the
  tumor geometry of the form \eqref{eq:front_perturbation} with
  $\epsilon = 0.05$ and for modes $k = 1, ..., 8$. We regard the mean
  of \eqref{eq:dispersion_relation_largeDout} during a selected time
  interval as an analytical prediction and measure the growth of each
  mode by fitting the amplitude to an exponential growth
  law. \figref{DLCM_Lambda_comparison} indicates that the rates agree
  fairly well, thus supporting the use of the PDE-based stability
  analysis in predicting the behavior also of the DLCM model.}
\begin{SCfigure}
    \includegraphics[scale=1]{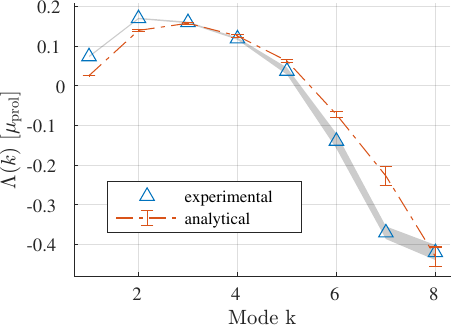}
    \caption{\review{Comparison between analytical
        \eqref{eq:dispersion_relation_largeDout} and experimental
        perturbation growth rates for the first few modes when
        $\sigma = 10^{-4}$. The error bars show the standard deviation
        of the analytical values during the time interval of
        measurement, and the shaded region corresponds to the standard
        deviation of the growth rate estimate during the same interval
        (20 independent runs).}}
  \label{fig:DLCM_Lambda_comparison}
\end{SCfigure}

% comparison between stochastic model and PDE
\review{Motivated by the quantitative evidence for a correspondence
  between the PDE analysis and the DLCM model, we} compare the
morphology and the growth patterns of the stochastic model and the PDE
model for similar parameters. \review{To avoid perturbations that are
  biased by the discretization method we add small amounts of white
  noise to the cell density updates. These and further details of the
  PDE solver, including the implementation of surface tension, are
  discussed in \S\ref{apx:numerical_methods}}. We evaluate the tumor
  boundary \textit{roundness} defined over a 2D region as
\begin{equation}
    \label{eq:roundness}
    \text{Roundness} = \frac{4\pi \times \text{Area}}{\text{Perimeter}^2},
\end{equation}
which ranges from $0$ to $1$, where $1$ means the shape is perfectly
circular and smaller values measure its deviation from circularity.

% (2b) 'Sharpness' study. 2D simulation at the interface between stability and
%instability. Low sigma -> uninhibited growth eventually.
We first investigate numerically the effects of surface tension on the
tumor morphology. Specifically, we study the onset of
second mode instability according to
\eqref{eq:dispersion_relation_largeDout} for a tumor growing in two
spatial dimensions. For this purpose, we compare the growth using
$\sigma = 2 \cdot 10^{-3}$ and $\sigma = 5 \cdot 10^{-4}$ until
$t = 30$, during which the former value is stable for $k=2$ although
it is not stable over larger times $t$. \review{For both experiments,
  we compare morphology and growth using a single simulation per model
  (a discussion on the impact of stochasticity is offered in
  \S\ref{sec:discussion}).}

\begin{figure}[!ht]
  \centering
  \begin{subfigure}[b]{.45\linewidth}
  \centering
    \includegraphics[scale=1]{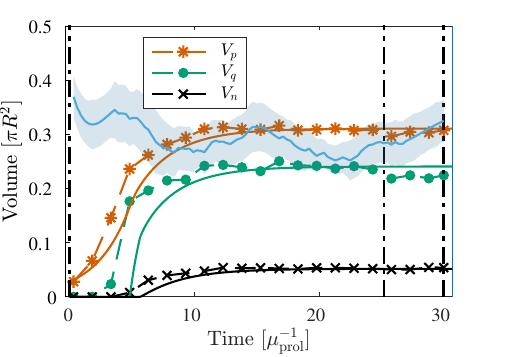}
    \caption{\label{fig:regional_evolution1_DLCM}}
  \end{subfigure}\qquad
  \begin{subfigure}[b]{.45\linewidth}
  \centering
    \includegraphics[scale=1]{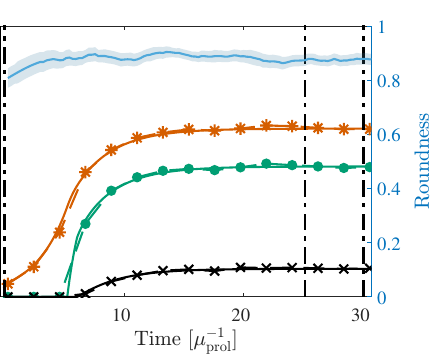}
    \caption{\label{fig:regional_evolution1_PDE}}
  \end{subfigure}
  \caption{\review{Evolution of characteristic volumes for DLCM (a)
      and PDE (b), respectively, using the standard parameters in
      \tabref{PDE_parameters_numerical} and
      $\sigma = 2 \cdot 10^{-3}$. The 2D solutions (dashed) are
      compared with the 1D solution
      \eqref{eq:PDE_regional_relations}--\eqref{eq:PDE_regional_dynamics}
      (solid) for the same parameters. Blue shows the roundness
      \eqref{eq:roundness} of the tumor boundary over time with moving
      window standard deviation in shaded. The vertical lines indicate
      the times $t\in [0,25,30]$ of the solutions shown in
      \figref{spatial_evolution1}.}}
  \label{fig:regional_evolution1}
\end{figure}

% DLCM vs. PDE
\begin{figure}
  \centering
  \begin{subfigure}[b]{.175\linewidth}
    \centering
    \includegraphics[scale=1]{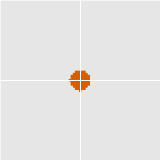}
    \caption{\label{fig:spatial_evolution1_DLCM_a}}
  \end{subfigure}
  \begin{subfigure}[b]{.175\linewidth}
    \centering
    \includegraphics[scale=1]{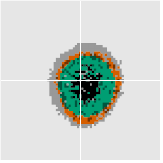}
    \caption{\label{fig:spatial_evolution1_DLCM_b}}
  \end{subfigure}
  \begin{subfigure}[b]{.175\linewidth}
    \centering
    \includegraphics[scale=1]{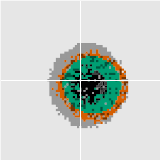}
    \caption{\label{fig:spatial_evolution1_DLCM_c}}
  \end{subfigure}
  
  \begin{subfigure}[b]{.175\linewidth}
    \centering
    \includegraphics[scale=1]{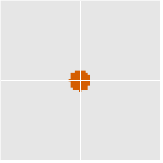}
    \caption{\label{fig:spatial_evolution1_PDE_a}}
  \end{subfigure}
  \begin{subfigure}[b]{.175\linewidth}
    \centering
    \includegraphics[scale=1]{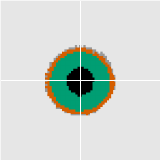}
    \caption{\label{fig:fig:spatial_evolution1_PDE_b}}
  \end{subfigure}
  \begin{subfigure}[b]{.17\linewidth}
    \centering
    \includegraphics[scale=1]{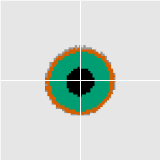}
    \caption{\label{fig:fig:spatial_evolution1_PDE_c}}
  \end{subfigure}
  \caption{\review{Solutions corresponding to
      \figref{regional_evolution1} at $t = 0$, $25$, $30$,
      respectively, with DLCM in the top row and the PDE in the bottom
      row. Color scheme as in \figref{PDE_creeping_example}. For DLCM,
      darker gray shows quiescent cells below $\kappadeath$, and
      darker shades of red and green indicate doubly occupied
      voxels.}}
  \label{fig:spatial_evolution1}
\end{figure}

% Experiments on sharpness of morphological stability analysis
\figref{regional_evolution1} shows the evolution of the tumor's
characteristic volumes for both models together with the tumor
roundness \eqref{eq:roundness} using the larger value of
$\sigma$. \figref{spatial_evolution1} shows snapshots of the solution
corresponding to \figref{regional_evolution1}. Both solutions remain
close to being radially symmetric during the full simulations since
the small second mode instability does not show during these
relatively short time intervals. Notably, the creeping effect becomes
apparent earlier for the DLCM simulations.

\begin{figure}[!ht]
  \centering
  \begin{subfigure}[b]{.45\linewidth}
  \centering
    \includegraphics[scale=1]{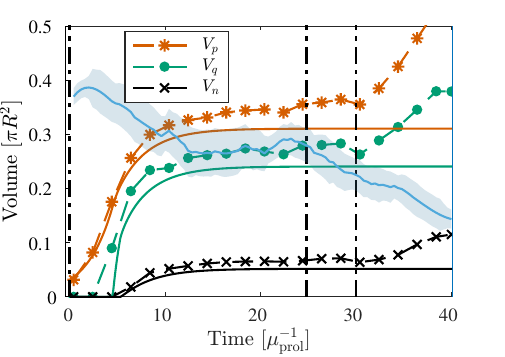}
    \caption{\label{fig:regional_evolution2_DLCM}}
  \end{subfigure}\qquad
  \begin{subfigure}[b]{.45\linewidth}
  \centering
    \includegraphics[scale=1]{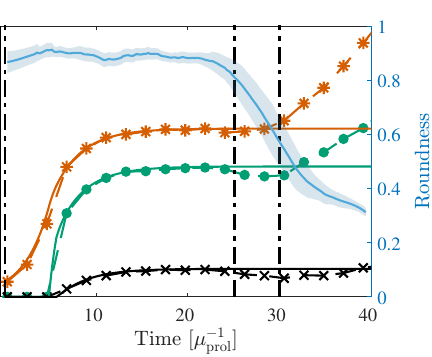}
    \caption{\label{fig:regional_evolution2_PDE}}
  \end{subfigure}
  \caption{\review{Solutions from DLCM (left) and PDE (right),
      respectively, using the standard parameters
      \tabref{PDE_parameters_numerical} and
      $\sigma = 5 \cdot 10^{-4}$, which is insufficient to ensure
      radial symmetry. The 2D solutions (dashed) are compared with the
      1D solution
      \eqref{eq:PDE_regional_relations}--\eqref{eq:PDE_regional_dynamics}
      (solid) using the same parameters and shown in the same units of
      time. The PDE model tumor splits into two parts just before the
      final time where the roundness metric is undefined. The vertical
      lines indicate the times $t \in [0, 25, 30]$ of the spatial
      solutions shown in \figref{spatial_evolution2}.}}
  \label{fig:regional_evolution2}
\end{figure}

% DLCM vs. PDE
\begin{figure}[!ht]
  \centering
        \begin{subfigure}[b]{.175\linewidth}
    \centering
    \includegraphics[scale=1]{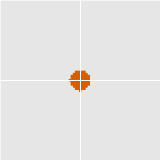}
    \caption{\label{fig:spatial_evolution2_DLCM_a}}
   \end{subfigure}
    \begin{subfigure}[b]{.175\linewidth}
    \centering
    \includegraphics[scale=1]{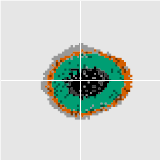}
    \caption{\label{fig:spatial_evolution2_DLCM_b}}
  \end{subfigure}
      \begin{subfigure}[b]{.175\linewidth}
    \centering
    \includegraphics[scale=1]{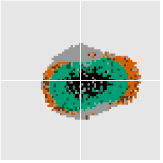}
    \caption{\label{fig:spatial_evolution2_DLCM_c}}
  \end{subfigure}%
  
        \begin{subfigure}[b]{.175\linewidth}
    \centering
    \includegraphics[scale=1]{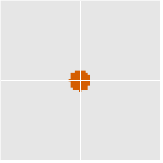}
    \caption{\label{fig:spatial_evolution2_PDE_a}}
  \end{subfigure}
        \begin{subfigure}[b]{.175\linewidth}
    \centering
    \includegraphics[scale=1]{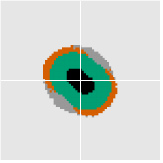}
    \caption{\label{fig:spatial_evolution2_PDE_b}}
  \end{subfigure}
        \begin{subfigure}[b]{.175\linewidth}
    \centering
    \includegraphics[scale=1]{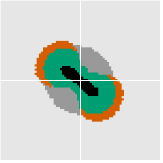}
    \caption{\label{fig:spatial_evolution2_PDE_c}}
  \end{subfigure}
    \caption{\review{Solutions corresponding to
    \figref{regional_evolution2} at $t = 0$, $25$, $30$, respectively.  Color scheme as in \figref{spatial_evolution1}}}
  \label{fig:spatial_evolution2}
\end{figure}

% ... same experiment but with lower sigma 
Similarly, \figref{regional_evolution2} and
\figref{spatial_evolution2} shows regional evolution and spatial
solutions, respectively, for the lower value of $\sigma$. Both
models display a significant decrease in roundness (accompanied by a
total volume increase) as the tumor begins to split in two some time
after the growth has plateaued. A notable difference between the
growth curves during this process is that the DLCM tumor does not
reach a fully stationary state before the splitting, possibly due to
the higher exposure to noise in the stochastic model.

% (2c) Examples of various morphologies from a range of parameters
Finally, \figref{various_spatial} shows the resulting morphology of
both models using four different and smaller values of $\sigma$. For
these experiments we use use a lower $\kappadeath = 0.92$ for a
thinner proliferating rim which results in effective parameters
$\mudeathbar = 1.0$ and $\lambdabar = 1.1$. Again, the radially
symmetric tumor displays significant creeping only for the DLCM model
at the selected final time (cf.~\figref{various_spatial_DLCM_a} and
\figref{various_spatial_PDE_a}). The morphologies of the tumors are
similar in terms of emergent unstable modes and sizes of the
characteristic regions, e.g., \review{the tumors beginning to separate
  into two is seen in both \figref{various_spatial_DLCM_b} and
  \figref{various_spatial_PDE_b}.} For the most unstable case using
$\sigma = 0$ in \figref{various_spatial_DLCM_d}, the DLCM tumor grows
somewhat larger and with different morphological and regional
characteristics. Finally, \figref{various_spatial_DLCM_a} and
\figref{various_spatial_DLCM_c} display small cell clusters detaching
from the tumor to grow on their own, a phenomena that we never observe
in our simulations of the PDE model
\review{(cf.~\figref{various_spatial_PDE_a} and \figref{various_spatial_PDE_c})}.

\begin{figure}[!ht]
  \centering
        \begin{subfigure}[b]{.175\linewidth}
    \centering
    \includegraphics[scale=1]{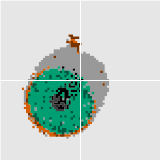}
    \caption{\label{fig:various_spatial_DLCM_a}}
   \end{subfigure}
    \begin{subfigure}[b]{.175\linewidth}
    \centering
    \includegraphics[scale=1]{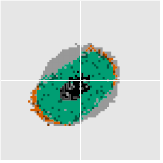}
    \caption{\label{fig:various_spatial_DLCM_b}}
  \end{subfigure}
      \begin{subfigure}[b]{.175\linewidth}
    \centering
    \includegraphics[scale=1]{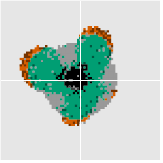}
    \caption{\label{fig:various_spatial_DLCM_c}}
  \end{subfigure}
        \begin{subfigure}[b]{.175\linewidth}
    \centering
    \includegraphics[scale=1]{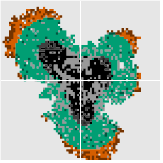}
    \caption{\label{fig:various_spatial_DLCM_d}}
  \end{subfigure}%
  
        \begin{subfigure}[b]{.175\linewidth}
    \centering
    \includegraphics[scale=1]{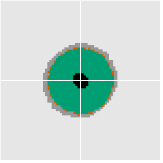}
    \caption{\label{fig:various_spatial_PDE_a}}
  \end{subfigure}
    \begin{subfigure}[b]{.175\linewidth}
    \centering
    \includegraphics[scale=1]{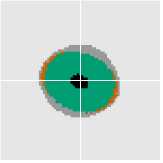}
    \caption{\label{fig:various_spatial_PDE_b}}
  \end{subfigure}
    \begin{subfigure}[b]{.175\linewidth}
    \centering
    \includegraphics[scale=1]{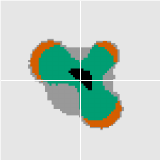}
    \caption{\label{fig:various_spatial_PDE_c}}
  \end{subfigure}
    \begin{subfigure}[b]{.175\linewidth}
    \centering
    \includegraphics[scale=1]{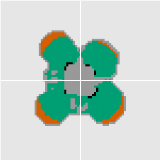}
    \caption{\label{fig:various_spatial_PDE_d}}
  \end{subfigure}
  \caption{Solutions to DLCM (top row) and the PDE (bottom row) for
    various parameters highlighting different aspects of the model in
    relation to \eqref{eq:dispersion_relation_explicit}. Left to right
    shows decreasing surface tension with $\sigma = 2 \cdot 10^{-3}$,
    $7 \cdot 10^{-4}$, $1 \cdot 10^{-4}$, $0$, and top row uses
    parameters $\kappaprol = 0.94$, $\kappadeath = 0.92$,
    $\mudeath = 0.5$, with bottom row using effective parameters
    $\mudeathbar = 1.0$, $\lambdabar = 1.1$. Solutions are shown at
    $t = 30$. Color scheme as in \figref{spatial_evolution1}}
  \label{fig:various_spatial}
\end{figure}
% end Experiment 2c

%**************************************************************************
\section{Discussion}
\label{sec:discussion}

% 0: summarize
We have analyzed the morphological stability of a PDE model of
avascular tumors. The PDE was derived to closely represent the
mean-field of a stochastic model expressed in the DLCM framework.
Assuming radial symmetry in the PDE model, we first characterized the
growth dynamics as well as the stationary state. A linear stability
analysis in two dimensions was subsequently carried out and we found a
dispersion relation that describes the stability of the morphology of
the tumor. Finally, we compared the analytical predictions with
numerical simulations of both the stochastic DLCM and the PDE
model. \review{The observed morphology of the stochastic model was
  found to be in line with the predictions from the PDE analysis,
  including also the proposed} relations required for a stable
stationary state.

% 1a: radial stability
The Saffman-Taylor instability acts on the tumor boundary, where the
determining factors are the porous medium permeability and the tissue
viscosity as summarized in the coefficients $D$ and $\Dout$. For a
transient tumor growth, these coefficients determine whether the tumor
boundary is stable ($\Dout > D$) or unstable ($\Dout < D)$. In the
former case, as the tumor growth slows down due to oxygen starvation,
perturbations on the boundary are amplified, thus destroying radial
symmetry unless the surface tension parameter $\sigma$ is large
enough; this was shown analytically in
\S\ref{subsec:analysis_stability_relations} and experimentally in
\S\ref{subsec:results_radial_symmetry}. This also explains the
asymmetry and unlimited growth of the original DLCM tumors presented
in \cite{engblom2018scalable}, which did not implement an explicit
surface tension effect.

% 1b: comparison with other models/analyses
\review{Morphological instability arising in conditions when nutrients
  are scarce is in line with analyses and simulations of other models
  of tumor growth and cell colonies. The model in
  \cite{cristini2003nonlinear} is a fluid-based PDE model including
  Darcy flow, with cell growth proportional to nutrient level, and a
  constant apoptosis rate. The model boundary velocity is explicitly
  dependent on the nutrient gradient, in contrast to the model
  analyzed herein that nevertheless displays a similar nutrient
  dependent growth instability. Further examples show similar
  instabilities in nutrient-deprived conditions both in agent-based
  models of cell colonies \cite{gerlee2007stability} and in hybrid
  models of tumor growth including cell cycles and ECM
  \cite{anderson2006tumor}. The former model tumor expands solely due
  to cell proliferation and not to pressure-driven migration, thus
  suggesting the persistence of this type of instability across a
  range of models.}

% 2: specific simulation observations

% 2a: creeping
Interestingly, our model predicts a creeping effect in which even an
otherwise stable tumor as a whole migrates towards the oxygen
source. Experiments in \S\ref{subsec:results_stochasticity} suggest
that the stochastic model has a somewhat stronger tendency for
creeping than the PDE model. The larger noise levels of the DLCM model
is a good candidate explanation for this difference. Moreover, the
creeping effect cannot be inhibited by surface tension, but must be
controlled through additional mechanisms of the model such as an
elastic response from the external tissue (see
\cite{walker2023minimal} for a review on minimal morphoelastic tumor
models). Thus, the creeping phenomena prompts the following questions:
Does creeping occur \textit{in vitro} or \textit{in vivo} and if so,
over what timescales? If not, what mechanisms keep it from occurring;
alternatively, \review{what assumptions could make our model more realistic
in this regard?}

% 2b: detached cell clusters
An additional observation from the DLCM model was the detachment of
cell clusters even in the case of a surface tension large enough to
support a radially symmetric solution. This is most likely due to the
discrete and random nature of individual cells in the
model. Detachment is therefore a distinct feature of on-lattice
stochastic modeling in this context, which appears to be a more realistic
representation of tumors growing under noisy conditions than a purely
fluid mechanical continuous model can offer.
%
% 2c: impact of stochasticity on region volumes and tumor size
\review{That said, the stochasticity of the DLCM model does not have a
  significant impact on the region volumes and final size of a
  radially symmetric tumor. Rather, assuming stability, these outcomes
  are fairly accurately governed by the deterministic relations
  \eqref{eq:PDE_regional_relations}--\eqref{eq:PDE_regional_dynamics}.
  The process noise does, however, have an impact on the morphology of
  the tumor under less surface tension, which implicitly affects the 
  tumor size first when the boundary has been significantly distorted.  
  A deeper investigation of this is outside the scope of the experiments
  reported here.}

% 2d: carrying capacity gives different behavior in DLCM w.r.t. PDE
One fundamental difference between the two models is the spatial
exclusion principle which is implemented in the DLCM framework via the
carrying capacity. A consequence of this can be seen when comparing
\figref{spatial_evolution2_DLCM_c} and
\figref{spatial_evolution2_PDE_c}, where the PDE tumor is close to
separating into two pieces, while the DLCM tumor in contrast retains
an oval shape. The latter is due to necrotic cells which degrade while
still occupying voxels, thereby slowing down the mass
flow. Similarly, \figref{various_spatial_DLCM_d} and
\figref{various_spatial_PDE_d} display significant differences in
morphology and size, where the former model supports a larger necrotic
region. These examples highlight emergent differences between the two
ways of modeling cell extent and migration and call for an input of
biological observations to approach a higher level of realism.

% 3: possible model additions/shortcomings
We end by briefly mentioning some potential modifications that may
improve on the expressive power of the model. Considering the PDE model
first, we see from \eqref{eq:PDE_pressure_proliferatingregion} that
the ambient pressure becomes very large for a stiff external
medium. Thus, when modeling such scenarios the addition of
pressure-dependent effects such as pressure-driven oxygen flow or a
pressure-based proliferation rate become relevant. Such additions
carry over to the stochastic framework in a fairly straightforward
manner. To further improve on the realism of the nutrient modeling, a
limit on the diffusion flux of the oxygen into the tumor across its
boundary could also be considered (cf.~\cite{folkman1973self}). While
these model modifications are readily implemented, the resulting
emergent behavior is not obvious and a precise mathematical analysis
is more involved due to the increase of nonlinear feedback mechanisms.

% 4: outro
In conclusion, our basic stochastic avascular tumor model turned out
to be a fruitful tool in leading to some new insights as well as
investigations of the effects of various coupling- and feedback
relations. Building and analyzing the mean-field PDE in tandem with
the discrete stochastic model forced us to think more deeply in terms
of trade-offs for continuum models of discrete cellular agents; this
approach limits model refinements to a certain extent since the
discrete and the continuous versions need to be consistent. We
anticipate that the combination of a mean-field PDE and a stochastic
model built from first principles will enable the development of
filtering tools aimed specifically at integrative Bayesian approaches
to data-driven applications. \review{For example, the stochastic model
  can be used to characterize the process noise of a Kalman filter
  that leverages the PDE model as its state transition.}

% Again from http://dx.doi.org/10.1016/j.semcdb.2015.07.001
% (brodland2015computational)
%
% Test hypotheses
% Lead to new insights
% Force us to think more deeply
% Suggest and refine experiments
% Interpret experiments
% Trace chains of causation
% Carry out sensitivity analyses
% Investigate coupling and feedback
% Integrate knowledge
% Lead to new approaches

\subsection{Availability and reproducibility}
\label{subsec:reproducibility}

The computational results can be reproduced with release 1.4 of the
URDME open-source simulation framework \cite{URDMEpaper}, available
for download at \url{www.urdme.org} \review{(see the avascular tumor
examples and the associated README in the DLCM workflow)}.

%**************************************************************************

\printbibliography[title={References}]

%**************************************************************************

\appendix
\clearpage

\section{Supporting perturbation results}
\label{apx:propositions_stability_analysis}

% section purpose
We consider a continuous quantity $q$ on two different circular
domains $\Omega_1, \Omega_2 \subseteq \mathbb{R}^2$, where $\Omega_2$
is a circle of radius $r$ inside $\Omega_1$, \review{both domains sharing origin}. We write
$\subquant{q}{1} = q|_{\Omega_1}$ and
$\subquant{q}{2} = q|_{\Omega_2}$. We assume $q$ to be continuous
across the closed interface between the two domains and that $q$ is
radially symmetric. We then introduce a small perturbation $\rpert{}$
of $r$, and thus $\rpert{} = \rpert{}(\theta)$ for
$\theta \in (-\pi,\pi]$, possibly inducing $q = q(s, \theta)$ when $q$
depends on $\rpert{}$ in some manner. Specifically, we let $q$ depend
on $\rpert{}$ via the continuity properties of $q$ across the
interface and study how $q$ responds to the perturbation in $r$.

% the perturbation eigenmodes are independent...
\begin{lemma}[\textit{Perturbation continuity}]
  \label{lemma:continuity}
  Let the interface be perturbed as
  \review{
  \begin{align}
    \label{eq:appendix_perturbed_interface}
    \rpert{}(\theta) &= r + \epsilon \zeta_k
                       \cos(k\theta),
  \end{align}
  for modes $k = 1, 2, ..., $}
 and write the induced perturbation on $q$ in the form
 \review{
  \begin{align}
    \label{eq:appendix_perturbation_quantities}
    \tilde{q}(s,\theta) &= q(s) +  \epsilon 
                           \phi_k(s) \cos(k\theta),
  \end{align}
  where $s$ is some radial position in $\Omega_1 \cup \Omega_2$.}
  We assume that the quantity $q$ is
  continuous across the interface. Then, to first order in $\epsilon$,
  \begin{align}
    \label{eq:appendix_perturbation_continuity}
    \subquant{q}{1}{'}(r)\zeta_k + \phi_k^{(1)}(r) = \subquant{q}{2}{'}
    (r)\zeta_k + \phi_k^{(2)}(r),
    \quad k \ge 1,
  \end{align}
  where $\phi_k^{(1)} = \phi_k|_{\Omega_1}$ and $\phi_k^{(2)} =
  \phi_k|_{\Omega_2}$, which is thus continuous across $r$ if $q$ is $C^1$
  continuous.
\end{lemma}

\begin{proof}
  The quantities $\pert{q}{1}$ and $\pert{q}{2}$ are continuous across
  $\rpert{}$,
  i.e.,
  \begin{equation}
    \pert{q}{1}(\rpert{}) =
    \pert{q}{2}(\rpert{}).
    \label{eq:appendix_quantity_continuity}
  \end{equation}
  We expand \eqref{eq:appendix_quantity_continuity} according to the
  assumed structure of the perturbations
 \eqref{eq:appendix_perturbed_interface} and
 \eqref{eq:appendix_perturbation_quantities} to
  get that 
  \review{
  \begin{eqnarray}
    q^{(1)}(r) + \epsilon q^{(1)}{'}(r) \zeta_k \cos(k\theta) &+& \epsilon 
                                                     \phi_k^{(1)}(r) \cos(k\theta) + \Ordo{\epsilon^2} = \nonumber \\ 
    q^{(2)}(r) + \epsilon q^{(2)}{'}(r) \zeta_k \cos(k\theta) &+& \epsilon 
                                                     \phi_k^{(2)}(r) \cos(k\theta) +
                                                     \Ordo{\epsilon^2}.
                                                     \label{eq:appendix_expanded_continuity}
  \end{eqnarray}
  }
  Now, the unperturbed quantities are assumed to be continuous across the
  unperturbed interface, i.e., $q^{(1)}(r) = q^{(2)}(r)$, \review{and we arrive at \eqref{eq:appendix_perturbation_continuity} to first
  order in $\epsilon$.}
\end{proof}

% comments on this
Lemma~\ref{lemma:continuity} extends in the obvious way in case there
are multiple convex interfaces and the perturbation is imposed on one
of them.

% ... and it holds true also if the derivative is continuous
We next show how the derivatives of the perturbation coefficients
relate.
\begin{lemma}[\textit{Perturbation $C^1$-continuity}]
  \label{lemma:continuous_derivative}
  Under the same conditions as in Lemma~\ref{lemma:continuity} and adding
  the assumption that $C^1$-continuity for $q$ also holds across the
  interface, \review{then to first order in $\epsilon$}
  \begin{align}
    \subquant{q}{1}{''}(r)\zeta_k+ \phi_k^{(1)}{'}(r) = \subquant{q}{2}{''}
    (r)\zeta_k +
    \phi_k^{(2)}{'}(r), \quad k \ge 1.
    \label{eq:appendix_perturbation_derivative_continuity}
  \end{align}
\end{lemma}

\begin{proof}
  The \review{directional} derivative of a perturbed quantity in polar coordinates at any
  point on the interface $\rpert{}(\theta)$ becomes
  \review{
  \begin{eqnarray}
    (\hat{n}_r, \hat{n}_{\theta}) \cdot \nabla 
    \left(q(\rpert{}) + \epsilon \phi_k(\rpert{})
    \cos(k\theta)\right) = \nonumber \\ 
    \label{eq:appendix_gradient}
    q'(\rpert{})\hat{n}_s + \epsilon \phi_k'(\rpert{})
    \cos(k\theta))\hat{n}_s - \frac{k}{\rpert{}}\epsilon
    \phi_k(\rpert{}) \sin(k\theta)\hat{n}_{\theta},
  \end{eqnarray}
  }
  where $\boldsymbol{n} = (\hat{n}_r, \hat{n}_{\theta})$ is the normal
  vector to $\rpert{}$ in polar coordinates. The angular component,
  $\hat{n}_{\theta}$, is proportional to
  $\epsilon$ and thus the third
  term in \eqref{eq:appendix_gradient} is $\Ordo{\epsilon^2}$. We use this
  \review{simplified expression of the directional derivative} for the $C^1$-continuity of
  $q^{(1)}$ and $q^{(2)}$ at $\rpert{}$, and expand both around $r$ to get
  \review{
  \begin{eqnarray}
    \subquant{q}{1}{'}(r) + \epsilon \subquant{q}{1}{''}(r) \zeta_k
    \cos(k\theta) + \epsilon \phi_k^{(1)}{'}(r)
    \cos(k\theta) + \Ordo{\epsilon^2} = \nonumber \\ \subquant{q}{2}{'}(r) +
    \epsilon \subquant{q}{2}{''}(r)
    \zeta_k \cos(k\theta) + \epsilon \phi_k^{(2)}{'}(r)
    \cos(k\theta) + \Ordo{\epsilon^2}.
  \end{eqnarray}
  }
  The derivatives are equal in the unperturbed case and again considering
  first order in $\epsilon$, \review{we finally arrive at
  \eqref{eq:appendix_perturbation_derivative_continuity} above.}
\end{proof}

% comment on this
These two lemmas are enough to derive the propagation of a known
perturbation on one interface through a system of $C^0$- and
$C^1$-continuous quantities connected at shared interfaces. We
continue by applying these lemmas to a perturbed oxygen field governed
by \eqref{eq:PDE_oxygen_law} with three interfaces.

\paragraph{Preliminary perturbation response}

% analysis of how a perturbation on the tumor boundary propagates to
% the two inner regions by perturbing the oxygen distribution
The pressure field \eqref{eq:PDE_pressure_law} is coupled to the
oxygen distribution obeying \eqref{eq:PDE_oxygen_law}. To analyze the
fully coupled system as is done in
\S\ref{subsec:linear_stability_analysis}, an estimate of the response
of an outer boundary perturbation into the inner regions via this
feedback relation is required. As in
\S\ref{subsec:analysis_radial_symmetry} we assume a circular tumor
with cell density $\rho$ approximately constant and equal to one, and
hence the oxygen field $c$ is initially given by
\eqref{eq:PDE_oxygen_radialsymmetry}. Let $\Omegaext$, $\Omega_p$,
$\Omega_q$, $\Omega_n$, $\subseteq$ $\mathbb{R}^2$, where $\Omega_n$
is a disk and the rest are annuli according to
\figref{tumor_region_schematic}. The outer domain $\Omegaext$ is the
annulus with small radius $r_p$ and large radius $R$. We write
$\subquant{c}{i} = c|_{\Omega_i}$ for $i \in
\{\text{ext},p,q,n\}$. The perturbations of the inner regions as
induced from an arbitrary \review{perturbation mode} on $r_p$ and the
continuity properties of $c$ across each interface are then covered by
the following result.

\begin{lemma}[\textit{Perturbation response}]
  \label{lemma:oxygen_perturbation_resonse}
  Let the outer tumor boundary $\interface{p}$ be perturbed by
  \review{
  \begin{equation}
    \rpert{p}(\theta) = \interface{p} + \epsilon
    \rcoef{p}\cos(k\theta),
    \label{eq:front_perturbation_appendix}
  \end{equation}
  }
  for some $\lvert \epsilon \rvert \ll 1$. Write the induced inner
  perturbations on the same form
  \review{
  \begin{align}
    \label{eq:inner_regions_perturbations_appendix}
    \begin{split}
      \rpert{q}(\theta) &= \interface{q} + \epsilon \rcoef{q}
      \cos(k\theta),
      \\
      \rpert{n}(\theta) &= \interface{n} + \epsilon \rcoef{n}
      \cos(k\theta),
    \end{split}
  \end{align}
  }
  each defined as the interface between the regions of different \review{cellular}
  growth rates according to \eqref{eq:PDE_source_term} with the oxygen field
  governed by
  \eqref{eq:PDE_oxygen_law}. Then the inner
  perturbations are to first order in $\epsilon$ given by
  \begin{align}
    \begin{split}
      \rcoef{q} &= \frac{1}{k \interface{q}^{k-1} \interface{p}^{k-1}} \times
      \frac{\interface{q}^{2k} - \interface{n}^{2k}}{\interface{q}^2 -
        \interface{n}^2} \times \frac{1 -
        \interface{p}^{2k}}{1-\interface{n}^{2k}} \times \rcoef{p}, \\
            \rcoef{n} &= \frac{\interface{n}^{k-1}}{\interface{p}^{k-1}}
            \times \frac{1 - \interface{p}^{2k}}{1 -
        \interface{n}^{2k}} \times \rcoef{p}.
      \label{eq:regional_perturbation_coefficients}
    \end{split}
  \end{align}
\end{lemma}

\begin{proof}
Similar to \eqref{eq:inner_regions_perturbations_appendix}, we write the
induced perturbation on the quantity (here the oxygen $c$) in each region $i
\in \{\text{ext},p,q,n\}$ as 
\review{
\begin{equation}
  \cpert{i}(r,\theta) = \subquant{c}{i}(r) +  \epsilon \ccoef{i}(r) \cos(k\theta).
    \label{eq:oxygen_perturbations}
\end{equation}
}
We find $\ccoef{i}(r)$ by inserting \eqref{eq:oxygen_perturbations}
into the oxygen equation \eqref{eq:PDE_oxygen_law}, using that the oxygen
consumption is constant in each region, and considering only the first order
terms in $\epsilon$:
\begin{equation}
    \label{eq:oxygen_perturbation_form}
    \ccoef{i}(r) = a_1^{(i)} r^k + a_2^{(i)} r^{-k},
\end{equation}
for some constants $a_1^{(i)}$ and $a_2^{(i)}$. Assuming regularity at
the origin we conclude that $a_2^{(n)} = 0$. The rest of the constants
are determined from the continuity properties at the interfaces as
follows.

% continuity
Firstly, the oxygen is $C^1$-continuous across each interface, i.e.,
\begin{align}
    \cextpert(\interface{p}) = \subquant{c}{p}(\interface{p}), \quad
    \subquant{c}{p}(\interface{q}) = \subquant{c}{q}(\interface{q}), \quad
    \subquant{c}{q}(\interface{n}) = \subquant{c}{n}(\interface{n}). \nonumber
\end{align}
and we may apply Lemma~\ref{lemma:continuity} and
\ref{lemma:continuous_derivative} to get that
\begin{align}
    \label{eq:oxygen_perturbations_continuity}
    \beta_k^{(\text{ext})}(\interface{p})  = \ccoef{p}(\interface{p}), \quad
    \ccoef{p}(\interface{q})  = \ccoef{q}(\interface{q}), \quad
    \ccoef{q}(\interface{n})  = \ccoef{n}(\interface{n}),
\end{align}
and 
\begin{align}
    \begin{split}
    \label{eq:oxygen_perturbations_derivative_continuity}
    \cextcoef{'}(\interface{p}) - \ccoef{p}{'}(\interface{p}) &= \lambda \rcoef{p}, \\
    \ccoef{p}{'}(\interface{q}) - \ccoef{q}{'}(\interface{q}) &= 0, \\
    \ccoef{q}{'}(\interface{n}) - \ccoef{n}{'}(\interface{n}) &= - \lambda \rcoef{n},
    \end{split}
\end{align}
where the right-hand sides were obtained by differentiation of the radially
symmetric oxygen field, as given by
\eqref{eq:PDE_oxygen_radialsymmetry}.

% oxygen thresholds
Secondly, the oxygen levels at the inner interfaces are known and must
equal the corresponding thresholds:
\begin{align}
    \label{eq:oxygen_threshold_relations}
    \pert{c}{q}(\interface{n}) = \kappadeath, \quad
    \pert{c}{q}(\interface{q}) = \kappaprol.
\end{align}
To each of \eqref{eq:oxygen_threshold_relations} we now apply
Lemma~\ref{lemma:continuity}. Considering again the first order terms,
we get
\begin{align}
    \begin{split}
    \label{eq:oxygen_perturbations_threshold_relations}
    \rcoef{q} \subquant{c}{q}{'}(\interface{q}) + \ccoef{q}(\interface{q}) &= 0, \\
    \rcoef{n} \subquant{c}{q}{'}(\interface{n}) + \ccoef{q}(\interface{n}) &= 0.
    \end{split}
\end{align}

% combining everything
Finally, combining
\eqref{eq:oxygen_perturbations_derivative_continuity},
\eqref{eq:PDE_oxygen_radialsymmetry}, and
\eqref{eq:oxygen_perturbations_threshold_relations}, we arrive at the
response \eqref{eq:regional_perturbation_coefficients}.
\end{proof}

% summarize findings
In summary, \eqref{eq:regional_perturbation_coefficients} specifies
how the tumor's inner interfaces $\interface{q}$ and $\interface{n}$
respond to an outer perturbation on $\interface{p}$ given that the
interfaces are defined by an oxygen field \eqref{eq:PDE_oxygen_law}.

%**************************************************************************

\section{Numerical methods}
\label{apx:numerical_methods}

\paragraph{PDE solution}
% numerical approach to solving the PDE: FEM for the pressure and the
% oxygen governed by the Poisson equation, FV for the advection of
% cell density.
The pressure and oxygen fields \eqref{eq:PDE_pressure_law} and
\eqref{eq:PDE_oxygen_law}, respectively, are readily solved by a
finite element method (FEM). We solve for each quantity using linear
basis functions and, for simplicity, using a lumped mass matrix. Solving
for the oxygen, however, requires special considerations as the oxygen
consumption depends on the oxygen level at each time step. To address
this, we solve for the oxygen iteratively within a smaller time scale,
denoted $\tau$, while holding $t_n$ and all other variables that
depend on $t_n$ constant. We solve the resulting time-dependent heat
equation implicitly using the Euler backward method, while treating
the source-term explicitly. We iterate in pseudo-time $\tau$ until
$\vert c^{\tau+1}-c^{\tau} \vert$ is less than some tolerance after
which we set $c(t_{n+1}) = c^{\tau+1}$. The inner iterations are given
by
\begin{equation}
  \frac{c^{\tau+1} - c^{\tau}}{\Delta \tau} - \Delta c^{\tau+1} = 
  \begin{cases}
    - \lambda \rho(t_n), & c^{\tau} \geq \kappa_{prol}
    \text{ on } \Omega_h(t_n) \\
    0, & \text{otherwise,}
  \end{cases}
\end{equation}
where $\Omega_h(t_n)$ is the discretized tumor domain at time $t_n$.

% advection and FV
The advection of cell density according to
\eqref{eq:PDE_conservation_law} is solved by an explicit finite volume
(FV) scheme using the pressure distribution at a given time step. We
consider a square Cartesian grid and an initial uniform cell density
of arbitrary shape placed within an otherwise empty domain where
$\rho = 0$. We keep track of the moving boundary by finding volumes
where $\rho < \rho_{thresh}$. We apply the pressure boundary condition
\eqref{eq:BC_Young-Laplace} on such volumes that are simultaneously
adjacent to volumes with $\rho \geq \rho_{thresh}$ (details on the
implementation of surface tension are found in the next
paragraph). The threshold was determined to ensure a front speed
consistent with the compatibility condition \eqref{eq:PDE_darcys}
during tumor growth and the value $\rho_{thresh} = 0.9$ was chosen to
reduce smearing effects close to the boundary and maintain
$\rho \approx 1$ across the tumor domain.  Furthermore, cell densities
outside the tumor boundary should not consume oxygen, and we therefore
set $\lambda = 0$ for volumes with $\rho < \rho_{thresh}$. Using a
first order upwind scheme to solve \eqref{eq:PDE_conservation_law}
including a noise term amounts to the following FV scheme \review{
\begin{eqnarray}
    \label{eq:FV_scheme}
    \rho_i^{n+1} &=& \rho_i^{n} + \Delta
    t F_{i}^n + \omega \vert \Delta
    t F_{i}^n \vert ^{1/2}
    \times \mathcal{N},\\
    F_{i}^n &=& - \frac{1}{V_i} \sum_{e_{ij}
    \in \partial V_i} \frac{p_j^n - p_i^n}{h_{ij}}e_{ij} R_{ij}^n + \Gamma_i^{n},  \nonumber \\
    R_{ij}^n &=& 
    \begin{cases}
        \rho_i^n, \quad \text{if }  p_i^n - p_j^n \geq 0 \\
        \rho_j^n, \quad \text{if }  p_i^n - p_j^n < 0 \nonumber \\
    \end{cases}
\end{eqnarray}
} at time $t_n$, where $V_i$ is the volume of the volume element,
$\partial V_i$ is the boundary of the volume, $e_{ij}$ is the unique
edge adjacent to volumes $V_i$ and $V_j$. The quantities $\rho_i^{n}$
and $\Gamma_i^{n}$ are the volume average quantities of $\rho$ and
$\Gamma$, respectively, at time $t_n$ taken over volume
$V_i$. Finally, to excite all perturbation modes of the boundary
\review{without discretization bias}, we have added scaled Gaussian
noise to the cell density change as the last term in
\eqref{eq:FV_scheme} dictates, with
$\mathcal{N} \sim \mathcal{N}(0,1)$ and $\omega = 0.025$. \review{The
  form of the noise is suggested by a Wiener process approximation of
  a Poissonian interpretation of the drift term in
  \eqref{eq:FV_scheme} at a system size $\omega^{-1/2}$. For our
  simulations we used a Cartesian discretization of the mesh with edge
  length $\Delta x = 0.02$ and
  $\Delta t = \min(\Delta x, 0.1/\max_i(\vert F_i^n \vert) )$, which
  yielded stable solutions across our experiments.}

% surface tensions implementation
\paragraph{Surface tension}

The curvature $C$ in the pressure boundary condition
\eqref{eq:BC_Young-Laplace} is evaluated numerically on the tumor
boundary by estimating the derivatives on the corresponding contour
lines. At each time step, we find the contour lines of the tumor
boundary where $\rho = \rho_{thresh}$ \review{for the PDE model, and
  where $u_i = 1$ for the DLCM model}. These \review{contour lines are
  then parameterized by $(x_{c}(s), y_{c}(s))$ for $s \in [0,S]$
  assuming periodicity, i.e.,
  $(x_{c}(s), y_{c}(s)) = (x_{c}(s+S), y_{c}(s+S))$, since they are
  closed curves}. We find the cubic spline of each coordinate such
that after differentiation we obtain the signed curvature
\begin{equation}
  C = \frac{x_{c}(s)'y_{c}(s)'' - x_{c}(s)'' y_{c}(s)'}{(x_{c}(s){'}^2 +
    y_{c}(s){'}^2)^{3/2}}.
  \label{eq:apx_curvature}
\end{equation}
We then calculate the Young-Laplace pressure difference at each
boundary volume by using the curvature given by
\eqref{eq:apx_curvature} at the contour point closest to the
volume. We include a simple threshold in our implementation to avoid
that a pressure boundary condition is applied to undesired
contours. Specifically, we neglect contour lines that consist of a
number of points less than a fraction $f_{\min} = 0.95$ of the largest
contour. Thus, holes within the tumor or clusters of cells outside the
tumor do not produce any surface tension. This prevents spurious
surface tension and pressure values, which would otherwise be present
in particular in the stochastic model.

\paragraph{DLCM modifications}

% Outline new features of the stochastic model as motivated by the
% physics of the PDE model
The physics of the PDE model motivates two features of the stochastic
mode not present in the original presentation
\cite{engblom2018scalable}, namely surface tension and pressure sinks
due to mass loss.

% 1a) 'Discrete' surface tension and 1->1 migration rule
Surface tension is implemented in the same manner in DLCM as for the
PDE solver. However, the exclusion of certain events in DLCM due to
the voxel carrying capacity introduces a bias associated with the
strength of surface tension. Therefore, we allow cells \review{in
  singly occupied voxels on the boundary to migrate into the
  population with rate coefficient equal to $D_1$}. Without inwards
migration, there is nothing to balance out the sporadic outwards
migration induced by surface tension on the inherently noisy
boundary. In particular, \review{allowing inwards migration} ensures
that higher levels of surface tension do not introduce a bias in the
average front speed of the tumor due to this noise. We note that this
bias is unavoidably present in the initial growth phase of the tumor
when no attractive, necrotic region exists, as can be observed in the
higher initial rate of growth of \figref{regional_evolution1_DLCM}
compared with \figref{regional_evolution2_DLCM}.

% 1b) Pressure sinks from 'discrete mass loss', and cell degradation
The PDE model expresses that mass loss due to cell death produces
pressure sinks that drive cell migration to replace the lost mass, and
that the pressure sinks are proportional to the rate of mass loss, see
\eqref{eq:PDE_source_term}. In the stochastic setting, cell mass loss
occurs for necrotic cells that are degrading at the rate $\mudeg$, and
for simplicity we use that the pressure sinks and pressure sources are
equal to $\mudeath$ and $\muprol$, respectively, akin to the
PDE. Further, to ensure that the degrading cells \review{are pushed by
  the population pressure} to the center of the necrotic region before
they have fully degraded, we derive an approximate value of the
degradation rate $\mudeg$ depending on $\mudeath$ for this to
happen. Consider a fixed, radially symmetric necrotic region and the
motion of a particle from the region's boundary into its center
\review{governed by \eqref{eq:PDE_darcys}. The pressure gradient at
  any point in this region is found from
  \eqref{eq:PDE_pressure_necroticregion} and letting the particle
  start at $r(t = 0) = r_n$, its position is given by}
\begin{equation}
    \label{eq:degrading_cell_motion}
    r(t) = r_n e^{-\mudeath t/2}.
\end{equation}
The expected time $\tau$ it takes for the particle to reach
$r(\tau) = \delta r_n, 0 \leq \delta < 1$ of the necrotic region's
center is thus readily obtained from
\eqref{eq:degrading_cell_motion}. The condition on $\mudeath$ that
ensures that the expected time for the particle to degrade is equal to
$\tau$ is then given by $1/\tau$ as
\begin{equation}
    \mudeg = \frac{\mudeath}{2\log(1/\delta)}. \nonumber
\end{equation}
For our implementation, we assume that $\delta = 0.01$ is enough to
ensure that migration of necrotic cells to the tumor center is
possible and occurs frequently.

\paragraph{\review{PDE effective parameters}}

\review{The effective parameters of the PDE model corresponding to the
  DLCM model parameters are derived by comparing the models radially
  symmetric solutions. There are three parameters that differ in value
  due to model differences: $\muprol$, $\mudeath$, and
  $\lambda$. Firstly, there is a difference in volumetric growth rates
  between the models since the DLCM tumor must undergo both a
  proliferation event and a migration event separately to expand in
  size. From a DLCM model simulation using the parameters in
  \tabref{PDE_parameters_numerical} and $\sigma = 0$, we consider the
  initial exponential growth phase ($r_n = r_q = 0$) such that the
  solution to \eqref{eq:PDE_front_velocity} reduces to
  $r_p(t) = r_p(0) \exp(\muprol/2 \, t)$. We fit observations of $r_p$
  during this phase by minimizing the mean-square error to get an
  estimate of the corresponding effective growth rate of the PDE
  model, and we find that $\muprolbar \approx \muprol/2.7$. Since the
  dimensionless parameter $\mudeathbar$ is proportional to
  $\muprolbar^{-1}$, the former is scaled accordingly. Secondly, we
  note that since doubly occupied voxels consume twice the amount of
  oxygen in DLCM, the effective $\lambda$ can be expected to be
  greater for the corresponding PDE assuming identical characteristic
  regions. The effective parameter $\bar{\lambda}$ is estimated by
  taking the mean of \eqref{eq:PDE_regional_relations} using the fixed
  DLCM model parameter $\kappaprol$ and observed values of
  $(r_n,r_q,r_p)$ during a stationary phase of the DLCM
  simulation. This gives the estimate $\bar{\lambda} \approx
  1.15$. The same procedure yields $\muprolbar = 1.0$ and
  $\bar{\lambda} = 1.1$ for the other set of experiments that uses different
  parameters shown in \figref{various_spatial}.}

\end{document}